\documentclass[11pt]{amsart}

\usepackage{graphicx}

\usepackage{graphics,color}

\usepackage{psfrag}

\usepackage{amssymb}

\usepackage{enumitem}

\setlist{itemsep=3pt plus 1pt minus 1pt, leftmargin=2em}

\usepackage[textsize=small]{todonotes}

\usepackage[colorlinks=false]{hyperref}

\newcommand{\setof}[1]{\left\{ #1 \right\}}

\newcommand{\FF}{\mathcal F}
\newcommand{\pair}[2]{\left< #1, #2\right>}
\newcommand{\pg}[2]{\pair{#1}{#2}_\G}
\newcommand{\pgn}{\pg{}{}}
\newcommand{\cl}[1]{\overline{#1}}
\newcommand{\norm}[1]{\left\Vert #1 \right\Vert}
\newcommand{\utau}{\underline\tau}

\newcommand{\inbetti}[2]{b_1(#1 \to #2)}

\DeclareMathOperator{\Comp}{Comp}

\DeclareMathOperator{\diam}{diam}

\newcommand{\numcomp}{\#\Comp}

\newcommand{\ceil}[1]{\lceil #1 \rceil}

\newcommand{\PSLR}{{\bf {PSL}}(2,\R)}

\newcommand{\PSLC}{{\bf {PSL}}(2,\C)}

\newcommand{\Ha}{{\mathbb {H}}^{2}}

\newcommand{{\Ho}}{ {\mathbb {H}}^{3}}

\newcommand{\R} {\mathbb {R} }

\newcommand{\Z} {\mathbb {Z}}

\newcommand{\N}{\mathbb {N}}

\newcommand{\C} {\mathbb {C}}

\newcommand{\wt}{\widetilde}

\newcommand{\wh}{\widehat}

\newcommand{\M}{{\bf {M}} }

\newcommand{\Su}{{\mathcal S}}

\newcommand{\dis}{{\bf d}}

\newcommand{\HH}{{\mathcal {H}}}
\newcommand{\hh}{\HH_h}
\newcommand{\hhh}{\HH'_h}

\newcommand{\Ne}{{\mathcal {N}}}

\newcommand{\G}{{\mathcal {G}}}

\newcommand{\la}{{\bf {\lambda}}}

\newcommand{\numb}{{\bf {n}}}

\newcommand{\gen}{{\bf {g}}}

\newcommand{\IM}{{\bf {i}}}

\newcommand{\Area}{\operatorname{Area}}

\newcommand{\len}{{\operatorname {l}}}

\newcommand{\Top}{{\operatorname {top}}}

\newcommand{\Geo}{{\operatorname {geo}}}

\newcommand{\Hull}{{\operatorname {Hull}}}

\newcommand{\inj}{{\bf{r}}}

\newcommand{\sys}{{\bf{sys}}}

\newcommand{\Dol}{{\mathcal {D}}}

\newcommand{\Tr}{{\mathcal {T}}}

\newcommand{\lab}{\psi}

\newcommand{\from}{\colon}

\newtheorem{corollary}{Corollary}[section]

\newtheorem{theorem}[corollary]{Theorem}

\newtheorem{lemma}[corollary]{Lemma}

\newtheorem{proposition}[corollary]{Proposition}

\newtheorem{claim}[corollary]{Claim}

\newtheorem{definition}[corollary]{Definition}

\newtheorem{conjecture}[corollary]{Conjecture}

\theoremstyle{remark}

\newtheorem{remark}{Remark}

\usetikzlibrary{calc}

\begin{document}

\title[random surfaces] {Geometrically and topologically random surfaces in a closed hyperbolic three manifold}

\author[Kahn]{Jeremy Kahn}

\address{\newline Mathematics Department \newline Brown University \newline Providence, RI 02912, USA}

\email{jeremy\_kahn@brown.edu}

\author[Markovi\'c]{Vladimir Markovi\'c}

\address{ \newline All Souls College  \newline University of Oxford  \newline United Kingdom }

\email{markovic@maths.ox.ac.uk}

\author[Smilga]{Ilia Smilga}

\address{\newline Mathematical Institute  \newline University of Oxford  \newline United Kingdom }

\email{ilia.smilga@normalesup.org}



\begin{abstract}  We study the distribution of geometrically and topologically nearly geodesic random surfaces in a closed hyperbolic 3-manifold $\M$. In particular, we describe $\PSLR$ invariant measures on the Grassmann bundle $\G_2(\M)$ which arise as limits of random minimal surfaces. It is showed that if $\M$ contains at least one totally geodesic subsurface then every topological limiting measure is  totally scarring (i.e. supported on the totally geodesic locus), while we prove that  geometrical limiting measures are never  totally scarring.
\end{abstract}

\maketitle

\let\johnny\thefootnote
\renewcommand{\thefootnote}{}

\footnotetext{ This work was supported by the \textsl{Simons Investigator Award}  409735 from the Simons Foundation}
\let\thefootnote\johnny

\section{Introduction} 

Fix a closed hyperbolic 3-manifold $\M$. We begin by defining  the notions of geometrically and topologically random surfaces in $\M$.
\begin{definition}\label{def-1} We let  $\Su$ denote the set of conjugacy classes of surface subgroups of $\pi_1(\M)$.  For $\epsilon>0$, we let $\Su_\epsilon$ denote the subset of $\Su$ consisting of conjugacy classes of quasifuchsian surface subgroups whose limit set is a  $(1+\epsilon)$-quasicircle.
\end{definition}
The following well known proposition follows from the works of  Uhlenbeck \cite{uhlenbeck} and Seppi \cite{seppi}.
\begin{proposition}\label{prop-us}
There exists a universal constant  $\wh{\epsilon}>0$ such that when $\epsilon<\wh{\epsilon}$ there exists a unique minimal surface in the homotopy class defined by $\Sigma\in \Su_\epsilon$, and this minimal surface is immersed in $\M$.
\end{proposition}
Thus, when $\epsilon<\wh{\epsilon}$ we can define the function $\Area_\M\from\Su_\epsilon\to (0,\infty)$ by letting $\Area_\M(\Sigma)$ be the area of the (unique) minimal surface in the corresponding homotopy class. In the remainder of the paper we assume that 
$\epsilon<\wh{\epsilon}$. 
\begin{definition}\label{def-2} Let $g\ge 2$, and $T>0$. We set
$$
\Su^\Top_\epsilon(g)=\{\Sigma\in \Su_\epsilon: \gen(\Sigma)\le g\},
$$
and
$$
\Su^\Geo_\epsilon(T)=\{\Sigma\in \Su_\epsilon: \Area_\M(\Sigma) \le T\},
$$
where $\gen(\Sigma)$ denotes the genus. (To alleviate the notation  we often write $\Su_\epsilon(g)$ instead of 
$\Su^\Top_\epsilon(g)$, and $\Su_\epsilon(T)$ instead of $\Su^\Geo_\epsilon(T)$).
\end{definition}
A typical element of $\Su_\epsilon(g)$ is called a topologically random surface, and 
a typical element of $\Su_\epsilon(T)$ a geometrically random surface. The study of topologically random surfaces was initiated by Kahn-Markovi\'c \cite{k-m-1}, while the study of geometrically random ones was  initiated by  Calegari-Marques-Neves \cite{c-m-n}, and  Labourie \cite{labourie}.

\subsection{The motivation} In the next subsection we define  probability measures on the Grassmann bundle $\G_2(\M)$ which capture the distribution of random surfaces. Topological and geometrical limiting measures inform our knowledge regarding the distribution of random surfaces. 
Establishing the uniqueness of such measures (or at least describing the range of such limits) would tell us how a random surface distributes in $\M$.

As observed in \cite{labourie}, the motivation for studying geometrical limiting measures comes from the classical case dealing with the limiting measures arising as limits (when $T\to \infty$) of measures supported on closed geodesics of length at most $T>0$. It was proved by Bowen and Margulis that in this case the limiting measure is unique and equal to the corresponding Liouville measure on the tangent bundle of $\M$. The motivation for studying topological limiting measures is clear.

\subsection{Invariant limiting measures}
Let $\G_2(\M)$ denote the Grassmann bundle  (the 2-plane bundle) over $\M$. We begin by explaining how each $\Sigma\in \Su_\epsilon$ yields a probability measure on $\G_2(\M)$. Let $f\from S_\Sigma\to \M$ be the minimal immersion representing the homotopy class of $\Sigma$ (here $S_\Sigma$ is the  Riemann surface determined by the requirement that $f$ is conformal and harmonic).  
Since $f$ is an immersion, we have the  induced map $f\from S_\Sigma \to \G_2(\M)$ (to simplify the notation, we use the same letter to denote both  maps).  Let $m_S$ denote the hyperbolic area measure on $S_\Sigma$ normalised to be a probability measure.  We  define the pushforward measure  
\begin{equation}\label{eq-1}
\mu(\Sigma)=f_*m_S
\end{equation}
on the Grassmann bundle $\G_2(\M)$. 
\begin{remark} One can define another probability measure $\mu_1(\Sigma)$ on $\G_2(\M)$ using the induced area form on the subsurface $f(S)$.  Unless the minimal surface $f(S)$ is totally geodesic, the measure $\mu_1(\Sigma)$ may not be the same as $\mu(\Sigma)$. However, when $\epsilon$ is small the minimal map $f$ is uniformly locally close to a M\"obius map. Thus, the distance between the  probability measures $\mu(\Sigma)$ and $\mu_1(\Sigma)$ uniformly tends to zero when $\epsilon\to 0$, which means that the limiting measures we define below do not depend whether we use $\mu(\Sigma)$ or $\mu_1(\Sigma)$ to define them.
\end{remark}
\begin{definition} We say that a probability measure $\mu^\Top_\epsilon$ on $\G_2(\M)$ is an $\epsilon$-topological limiting measure  if there exists a sequence $g_n\to \infty$ such that 
$$
\mu^\Top_\epsilon=\lim\limits_{g_{n}\to \infty} \frac{1}{|\Su_\epsilon(g_n)|} \sum_{\Sigma\in \Su_\epsilon(g_n)} \mu(\Sigma).
$$
We say that  a probability measure $\mu^\Geo _\epsilon$ on $\G_2(\M)$ is an $\epsilon$-geometrical limiting measure if there exists a sequence $T_n\to \infty$ such that 
$$
\mu^\Geo_\epsilon=\lim\limits_{T_{n}\to \infty} \frac{1}{|\Su_\epsilon(T_n)|} \sum_{\Sigma\in \Su_\epsilon(T_n)} \mu(\Sigma).
$$
\end{definition}

\begin{definition} We say that a probability measure $\mu^\Top$ on $\G_2(\M)$ is a topological limiting measure  
if there exists a sequence $\epsilon_n\to 0$ such that 
$$
\mu^\Top=\lim\limits_{n\to \infty} \mu^\Top_{\epsilon_{n}},
$$
where $\mu^\Top_{\epsilon_{n}}$ is some $\epsilon_n$-topological limiting measure.
Likewise, we say that  a probability measure $\mu^\Geo$ on $\G_2(\M)$ is a geometrical limiting measure if there exists a sequence $\epsilon_n\to 0$ such that 
$$
\mu^\Geo=\lim\limits_{n\to \infty} \mu^\Geo_{\epsilon_{n}},
$$
where $\mu^\Geo_{\epsilon_{n}}$ is some $\epsilon_n$-geometrical limiting measure.
\end{definition}

\subsection{The main results} Each topological or geometrical limiting measure is invariant under the natural left  
$\PSLR$ action on $\G_2(\M)$. The classical Ratner theorem says that each such measure is a linear combination of the probability Liouville measure on $\G_2(\M)$, and the  probability measures supported on the images of totally geodesic surfaces in $\G_2(\M)$.  In particular, if $\M$ does not contain a totally geodesic subsurface then there exists 
a unique limiting measure (either topological or geometrical) and it is equal to the Liouville measure on $\G_2(\M)$.
Therefore, it remains to understand the case when $\M$ contains totally geodesic subsurfaces.
\begin{definition} Let $\mu$ denote a $\PSLR$-invariant probability measure on $\G_2(\M)$. Then by $\mu_{\mathcal{L}}$ we denote its Liouville part. Moreover, we say that $\mu$  is totally scarring if $\mu_{\mathcal{L}}$ is the zero measure.
\end{definition}

\begin{theorem}\label{thm-main-geo} Suppose $\M$ is a  closed hyperbolic 3-manifold. There exists a constant $0<q$ such that if $\mu^\Geo$ is a geometrical limiting measure then $q<|\mu^\Geo_{\mathcal{L}}|$. 
\end{theorem}

The previous theorem says that a geometrical limiting measure is never totally scarring. The opposite is true for topological limiting measures as the following theorem attests.

\begin{theorem}\label{thm-main-top} If $\M$ is a closed hyperbolic 3-manifold with a totally geodesic surface then every topological limiting measure is totally scarring. 
\end{theorem}

We finish this subsection with two conjectures.

\begin{conjecture}  \label{conj-geo}
Every geometrical limiting measure is Liouville.
\end{conjecture}

\begin{conjecture}  \label{conj-top}
Suppose that $\M$ is a closed hyperbolic 3-manifold with a totally geodesic surface. Then every topological  limiting measure is supported on $\HH_{h_{0}}$ where $h_0=h_0(\M)$ is the genus of the  totally geodesic surface of the smallest genus in 
$\M$.
\end{conjecture}

Answering these two conjectures  would complete the  description of  how random surfaces distribute in $\M$.

\begin{remark} The proof of Theorem \ref{thm-main-top} should go through in any symmetric space without many changes. The proof of Theorem \ref{thm-main-geo} should likewise go through given the analog of Rao's theorem. Conjectures \ref{conj-geo} and \ref{conj-top} should presumably apply as well in this more general setting  (see  Kassel's survey article  \cite{kassel}).
\end{remark}

\begin{remark}
Labourie  \cite{labourie} proved that the Liouville measure on $\G_2(\M)$ is the limit of $\mu_{f_n}$ for some sequence of $\epsilon_n$-nearly geodesic minimal (possibly disconnected) surfaces $f_n\from S_n\to \M$. Lowe and Neves \cite{l-n} showed that one can find such a sequence consisting of connected surfaces. Al Assal \cite{assal} then proved that the same is true for any convex combination of invariant probability measures on $\G_2(\M)$.
\end{remark}

\begin{remark} We expect both Theorem \ref{thm-main-geo} and Theorem \ref{thm-main-top} to hold if instead of conjugacy classes $\Su_\epsilon$ one considers commensurability classes of $(1+\epsilon)$-quasifuchsian sufraces in $\M$.
\end{remark}

\subsection{The totally geodesic locus} We let $\HH \subset \G_2(\M)$ be the union of the totally geodesic subsurfaces. For any $h \in \N$, 
we let $\hh \subset \HH$ be the union of the totally geodesic subsurfaces of genus at most $h$, 
and let $\hhh = \HH \setminus \hh$. We let $h_0(\M)$ be the genus of the least genus totally geodesic surface in $\M$. 

\begin{remark}
In Theorem \ref{thm-main-top} we can  prove the explicit estimate  $\mu^\Top(\hhh) \le 1/(C \log h)$ for large  $h$, and constant $C=C(\M)$.
\end{remark}

By the results of Margulis-Mohammadi \cite{m-m}, and Bader-Fisher-Miller-Stover \cite{b-f-m-s},  we know that a non-arithmetic hyperbolic 3-manifold has at most finitely many totally geodesic subsurfaces. We observe that in this case 
$\HH=\HH_h$ for some $h$.

\subsection{Outline of the proof of Theorem \ref{thm-main-geo}}
The proof of this theorem follows readily from two key lemmas. In the first lemma (which is Lemma \ref{lemma-comp} below) we show that there exists a constant $C_1>0$ such that a  geometrically random surface $\Sigma \in \Su_\epsilon(T)$ satisfies the bound
\begin{equation}\label{eq-l-1}
4\pi(\gen(\Sigma)-1) \ge (1+C_1\epsilon^2)T
\end{equation}
when $\epsilon$ is small enough. Here $C_1$ is a constant which only depends on $\M$ (and in particular it does not depend on $\epsilon$ or $T$).
Note that if $\Sigma$ is a totally geodesic surface of area $T$ then  $4\pi(\gen(\Sigma)-1) =T$. Therefore, the inequality (\ref{eq-l-1}) shows that the genus of a geometrically random surface is significantly larger than the genus of a totally geodesic surface of the same area. This is the key factor in the proof of  Theorem \ref{thm-main-geo}, and the key difference between geometrically and topologically random surfaces.

The second key lemma bounds from above the genus of a geometrically random surface $\Sigma\in \Su_\epsilon(T)$ in terms of the percentage of its area which
concentrates near the totally geodesic locus $\HH_h$.  In Lemma \ref{lemma-comp-1}, for every $h\in \N$, and every $\epsilon<\epsilon_0(h)$, we establish the estimate
\begin{equation}\label{eq-l-2}
4\pi(\gen(S)-1)\le \left(1+C_2\left(1-\frac{|S(\delta_1(h),h)|}{|S|}\right)\epsilon^2 \right)T.
\end{equation}
Here $f\from S\to \M$ is the $\epsilon$-nearly geodesic minimal surface representing the class $\Sigma$, and $S(\delta_1(h),h)$ is the subset of $S$ such that $f\big(S(\delta_1(h),h)\big)$ is close to the geodesic locus $\HH_h$. Note that the constant $C_2$ does not depend on $h$. 

Combing the two lemmas yields the estimate 
$$
0<\frac{C_1}{C_2} \le |\mu^\Geo_{\mathcal{L}}|
$$
which then proves Theorem \ref{thm-main-geo}.

\subsection{Outline of the proof of Theorem \ref{thm-main-top}}

Before outlining the basic idea of the proof we briefly recall the proof of the upper bound on the number of homotopy classes of genus $\gen$ surfaces in $\M$. Let $\Dol$ be a finite cover of $\G_2(\M)$ by open sets of sufficiently small diameter. 

Given a $\pi_1$-injective map $f\from S\to \M$ of a closed surface of genus $\gen$,
we can realize it as a pleated surface and find a bounded geometry triangulation of $S$.
We can then take a spanning tree $\wh\tau$ of the triangulation, 
along with a set of another $2\gen$ \emph{distinguished} edges, denoted by $e(\tau)$, so that the inclusion of $\wh\tau \cup e(\tau)$ into $S$ is surjective on $H_1$ and hence on $\pi_1$. For each vertex $v$ of $\tau$, we choose $D \in \Dol$ such that $f(v) \in D$ and denote $D$ by $\lab(v)$. We thereby obtain a \emph{$\Dol$-valued $\gen$-polygonalization} consisting of the pair $(\tau, \lab)$ where $\tau = (\wh\tau, e(\tau))$ is a \emph{graph pair} that we call a $\gen$-polygonalization because it efficiently divides a genus $\gen$ surface into polygons. 

The point is that we can reconstruct the homotopy class of $f$ from the $\Dol$-valued $\gen$-polygonalization,
where we think of $\tau = (\wh\tau, e(\tau))$ as an abstract graph (with a cyclic ordering of the edges at each vertex), and hence as purely combinatorial data. We then need only bound the number of possible such data, and in that way we obtain the $(Cg)^{2g}$ upper bound in \cite{k-m-1}.

In this paper%
\footnote{In this introduction we assume that $\M$ is non-arithmetic so $\HH = \hh$ for some $h$; the proof in the general case is very similar.}
 we refine this bound in the case where $f$ is $\epsilon$-nearly geodesic and at least $q$ of the induced measure $f_*m_S$ is away from $\HH$, to obtain a bound of $(cg)^{2g}$, where $c \equiv c(q, \epsilon) \to 0$ as $\epsilon \to 0$ and $q > 0$ remains fixed. 
To do so we let $\Dol$ be a system of ``wafer-thin disks'' that efficiently cover $\G_2(\M)$. 
Namely, 
a (two-dimensional) disk in $\M$ lifts to a disk in $\G_2(\M)$,
and each wafer-thin disk will be the thickening of such a disk in $\G_2(\M)$ by a small orthogonal radius $\eta$ to form an open set in $\G_2(\M)$. 
These wafer-thin disks all have small volume, 
but travel between them is carefully restricted for an $\epsilon$-nearly geodesic surface 
(with $\epsilon$ small given $\eta$):
if two points on such a surface are at bounded distance, lying respectively in $D_0, D_1 \in \Dol$,
then $D_0$ and $D_1$ are \emph{mutually compatible}, and there are only a bounded number of disks in $\Dol$ that are compatible with a given disk, where this bound is independent of the choice of $\eta > 0$.  

Then given a nearly geodesic $f\from S \to \M$ of genus $\gen$,
we 
\begin{enumerate}
\item
lift $f$ to $f\from S \to \G_2(\M)$,
\item
find a bounded triangulation and spanning tree for $S$ just as in \cite{k-m-1},
\item
find our  $2\gen$ distinguished edges so that at least $q\gen$ of them are some distance from $\HH$, and
\item
assign, 
as before,
an element $\psi(v) \in \Dol$ to each vertex $v$ so that $f(v) \in \psi(v)$.
\end{enumerate}
Our additional constraint in Step 3 on the choice of distinguished edges implies that the endpoints of the $q\gen$ edges (that are away from $\HH$) are roughly evenly distributed in $\G_2(\M)$, and hence with respect to $\Dol$. (This follows from Ratner's theorem.) The small volume of each disk in $\Dol$, combined with this equidistribution, implies that only a small proportion of these endpoints can be in any $D \in \Dol$. 
So we obtain a $\Dol$-valued $\gen$-polygonalization, along with some constraint on how many $v$ can satisfy $\lab(v) = D$ for a given $D \in \Dol$,  and for which the vertices of a given edge lie in mutually compatible elements of $\Dol$. We now need only bound the number of such combinatorial structures, and these constraints are sufficient to obtain the desired bound of $(c\gen)^{2\gen}$, where $c$ is small when $\epsilon$ is. 

On the other hand, since $\HH\neq \emptyset$, the M\"uller-Puchta \cite{m-p} estimate implies that the number of \emph{totally geodesic} surfaces in $\M$ of genus $\gen < g$ grows as $(Cg)^{2g}$. (These are all just covers of a component of $\HH$.) 
We can now compare this to the number of nearly geodesic surfaces where a definite part of the area lies \emph{away} from $\HH$,
and see that the latter number is much smaller than the former. 
Thus the only significant contributors to the measure for a topologically random nearly geodesic surface of genus at most $g$ are surfaces where most of their area lies close to $\HH$. This implies Theorem \ref{thm-main-top}.

\subsection{How to read this paper}
After Section \ref{sec-prelim}
the paper is divided into two halves; the first half proves Theorem \ref{thm-main-geo}, and the second half proves Theorem \ref{thm-main-top}.
Section \ref{sec-prelim}  provides some preliminaries and background, 
including a smoothing lemma and isoperimetric theorem for subsurfaces of a hyperbolic surface. 
After Section \ref{sec-prelim}, the two halves can be read in either order, starting with Sections \ref{section-lemmas} or \ref{sec-main-top}.
Both halves have a top-down structure, so the main theorems are proven in one section assuming one or more lemmas, which are then proven in the subsequent sections.

\section{Preliminaries} \label{sec-prelim}

Once and for all we fix a  closed hyperbolic 3-manifold $\M$. In this section we discuss the  Grassmann bundle of $\M$, and define and discuss elementary properties of $\epsilon$-nearly geodesic minimal surfaces and how they relate to the totally geodesic locus $\HH$ in $\M$.
In particular, we fix two constants $\delta_0(h)$ and $\delta_1(h)$ which capture the geometry of the locus $\HH_h$ in $\M$ and its interaction with nearly isometric minimal maps.

\subsection{Notation and conventions} If we define a constant without specifying what it depends on, it means that it only depends  on
$\M$.  If $X$ is a hyperbolic manifold (which in this paper is either a surface or a 3-manifold) then $d_X$  denotes the hyperbolic distance on $X$.  We use $|Y|$ to denote the hyperbolic area of a subset $Y\subset X$, and  $|\alpha|$ the hyperbolic length of a smooth curve $\alpha\subset X$. Furthermore, $B_r(x)$ denotes the ball of radius $r$ centred at $x$, and  $\Ne_r(Y)$  the $r$-neighbourhood of a set $Y$.

\subsection{The Grassmann bundle} 
The Grassmann bundle  $\G_2(\M)$ is the 2-plane bundle over $\M$. A point in $\G_2(\M)$ is a pair  $(x,\Pi)$ where $x\in \M$, and $\Pi<T_x\M$ a 2-plane. The universal cover of $\G_2(\M)$ is $\G_2({\Ho})$. By $\xi\from\G_2(\M)\to \M$ we denote the obvious fibration. The Levi-Civita connection on $T\M$ induces a connection on $\G_2(\M)$ and hence a canonical complement in $T\G_2(\M)$ to the fibers of $T\xi\from T\G_2(\M)\to T\M$. The standard Riemannian metric $\pgn$ on $\G_2(\M)$ is characterised by the following properties:
\begin{enumerate}
\item  
within each fibre $\pgn$ coincides with the usual Riemannian metric on 
$\G_2^3 = \R P^2$,
\item the map $\xi$ is a Riemannian submersion, and
\item
the fibres described in (1) are orthogonal to their canonical complements.  
\end{enumerate}
These properties imply that the canonical complement to the fiber at each point of $T\G_2(\M)$ is mapped isometrically by $T\xi$; this in turn implies that the lift of a path in $\M$ to $\G_2(\M)$ by parallel transport (with some starting 2-plane at the initial point) is an isometric lift.  

Obviously the lift of this metric to $\G_2({\Ho})$ is invariant under the action by $\text{Isom}({\Ho})$, and
agrees with the hyperbolic metric on totally geodesic planes in
$\G_2({\Ho})$. Moreover, for each $(x_1, \Pi_1), (x_2, \Pi_2) \in \G_2({\Ho})$, we have
$$
\dis_{\Ho}(x_1,x_2)\le \dis_\G\big((x_1,\Pi_1),(x_2,\Pi_2) \big).
$$

\subsection{The geodesic locus} \label{subsec-geod}
Given $h\in \N$, we let  $\HH_h\subset \G_2(\M)$ denote the union of totally geodesic surfaces of genus at most $h$.  We refer to $\HH_h$ as the totally geodesic locus of $\M$ of level $h$.  By $\wt{\HH}_h\subset \G_2({\Ho})$ we denote the preimage of $\HH_h$ by the covering from $\G_2({\Ho})$ to $\G_2(\M)$. We observe that each connected component of $\wt{\HH}_h$ is a geodesic plane.
\begin{proposition}\label{prop-delta} 
There exists $\delta_0(h)>0$ such that
\begin{itemize}
\item
$|\cl{\Ne_{\delta_0(h)}(\hh)}| \le \frac12 |\G_2(\M)|$ and 
\item
$\dis_\G(P_1,P_2)\ge 10\delta_0(h)$ when $P_1$ and $P_2$ are two different components of $\wt{\HH}_h$.
\end{itemize}
\end{proposition}
\begin{remark} In particular,  the geometry of the components of $\Ne_{\delta_{0}(h)}(\HH_h)$ depends only on $\delta_0(h)$ and the geometry of the components of $\hh$. 
\end{remark}

\begin{proof} 
The first property of $\delta_0(h)$ holds because $\HH_h$ is a closed subset of $\G_2(\M)$ of measure 0. 
The second property follows from the fact that $\HH_h$ has finitely many components, each of which is compact. 
\end{proof}

\subsection{Nearly geodesic minimal maps }

We begin with a definition.
\begin{definition}
We say that $f\from S\to \M$ is an $\epsilon$-nearly geodesic minimal map if:
\begin{enumerate}
\item $S$ is a closed hyperbolic surface, 
\item $f$ is a minimal map,
\item $f_*\from\pi_1(S)\to \pi_1(\M)$ is a quasifuchsian surface subgroup
whose limit set is a $(1+\epsilon)$-quasicircle.
\end{enumerate}
\end{definition}
Since every $\epsilon$-nearly geodesic minimal map $f\from S\to \M$ is an
immersion assuming $\epsilon$ is small,  the induced map  $f\from S\to
\G_2(\M)$ is well defined. We now define the subset of the minimal
surface which is close to the geodesic locus $\HH_h$.
\begin{definition}\label{def-deltaset}
Let $f\from S\to \M$ be an $\epsilon$-nearly geodesic minimal map. For
$\delta>0$, we let
$$
S(\delta,h)=\{p\in S: \dis_\G(f(p),\HH_h)<\delta  \}.
$$
\end{definition}

Next, we observe that  $\epsilon$-nearly geodesic minimal maps are nearly locally isometric, meaning that the restriction of $f$ to any ball of a bounded radius is close to an actual isometry. The following proposition readily follows from the compactness principle for minimal maps which is stated as Theorem 2.5  in \cite{h-l-s}.
\begin{proposition}\label{prop-near} For every $\eta>0$ there exists a constant $\epsilon(\eta)>0$ with the following properties. Let $f:S\to \M$ be a  $\epsilon$-nearly geodesic minimal map where  $\epsilon<\epsilon(\eta)$, and let $\iota:B_{10}(x_0)\to \M$ be a local isometry such that $\iota(x_0)=f(x_0)$ as points in $\G_2(\M)$.
Then   $\dis_{\G}(\iota(x),f(x))<\eta$ for every $x\in B_{10}(x_0)$.
\end{proposition}

We finish the section by proving the following proposition relating sets $S(\delta,h)$
for different values of $\delta$. 
\begin{proposition} \label{prop-prva} 
For all $\delta > 0$, and $h\in \N$,  there exists $\delta'=\delta'(\delta,h) > 0$ and $\epsilon_0=\epsilon_0(\delta,h) > 0$ such that for every $\epsilon$-minimal surface $f\from S\to \M$,
\begin{equation}
\Ne_{10}\big(S(\delta',h)\big) \subset S(\delta,h)
\end{equation}
assuming $\epsilon<\epsilon_0$.
\end{proposition}
\begin{remark}\label{remark-prva} For future reference, we define $\delta_1(h)=\delta'(\delta_0(h),h)$, and $\epsilon_0(h)=\epsilon_0(\delta_0(h),h)$.
\end{remark}

\begin{remark}
We could actually make $\delta'$ and $\epsilon_0$ independent of $h$, but in our context $\delta$ will depend on $h$, so we state and prove the proposition in its current form.
\end{remark}

\begin{proof} The proof  is by contradiction.  Fix $\delta>0$, and $h\in \N$, and assume that for each $n\in \N$ there exists $\epsilon_n$-nearly geodesic minimal map $f_n:S_n\to \M$, where $\epsilon_n\to 0$, and  sequences of points $x_n,y_n\in S_n$, such that 
\begin{equation}\label{eq-rake}
\dis_\G\big(f_n(x_n),\HH_h\big) \ge \delta,\quad \dis_\G\big(f_n(y_n),\HH_h\big) \le \frac{1}{n},
\end{equation}
and 
\begin{equation}\label{eq-rake-1}
d_{S_{n}}(x_n,y_n)\le 10.
\end{equation}
From \eqref{eq-rake}  we find that $f_n(x_n)\to (x,\Pi)\notin \HH_h$, and $f_n(y_n)\to (y,\Pi')\in \HH_h$. Combining  \eqref{eq-rake-1} with the assumption $\epsilon_n\to 0$, and by applying Proposition \ref{prop-near}, we conclude that there exists a totally geodesic plane $P\subset \G_2(\M)$ which contains both $(x,\Pi)$ and $(y,\Pi')$. This implies that $(x,\Pi)\in \HH_h$. This contradiction proves  the proposition. \end{proof}

\subsection{A smoothing lemma for subsets of $S$}
Here's a simple version of the Vitali Covering Lemma that will suffice in this paper:
\begin{lemma} \label{lem-vitali}
Suppose $X$ is a totally bounded metric space (e.g. a subset of a compact metric space), and $r > 0$. 
Then we can find points $x_1, \ldots x_m$ such that the balls $B_r(x_i)$ are pairwise disjoint, and $X \subset \bigcup_i B_{2r}(x_i)$. 
\end{lemma}
\begin{proof}
Since $X$ is totally bounded, there is a finite set $x_1, \ldots, x_m$ such that the $B_r(x_i)$ are pairwise disjoint, and this set cannot be extended to a larger one with the same property. Take any $x \in X$; there must be an $x_i$ such that $B_r(x_i) \cap B_r(x)$ is nonempty, and then $x\in B_{2r}(x_i)$.
\end{proof}

We can use \ref{lem-vitali} to prove a smoothing lemma for subsurfaces of a closed hyperbolic surface $S$:
\begin{lemma} \label{lem-smoothing}
Suppose $R\subset S$. Then we can find $R'$ with piecewise smooth boundary such that $R \subset R' \subset \Ne_1(R)$,
and $|\partial R'| \le 25 |\Ne_1(\partial R)|$. 
\end{lemma}

\begin{proof}
In what follows, for any $r>0$,  we let $|\partial B_r| = 2 \pi \sinh r$ denote the perimeter of a hyperbolic ball in $\mathbb{H}^2$ of radius $r$,
and we let $|B_r| = 2 \pi(\cosh r - 1)$ denote its area.  We will estimate some of these values with whole numbers.

By Lemma \ref{lem-vitali} we can find $x_1, \ldots, x_m \in \partial R$ such that the $B_{1/2}(x_i)$ are pairwise disjoint, but $\partial R \subset Q$, 
where $Q := \bigcup_i B_1(x_i)$. We let $R' = R \bigcup Q$, and observe that $R \subset R' \subset \Ne_1(R)$,  $\partial R' \subset \partial Q$ and $Q \subset \Ne_1(\partial R)$. 
Moreover
\begin{equation} \label{eq-qmb}
|\partial Q| \le m |\partial B_1|.
\end{equation}
If $y \in B_1(x)$, then $B_{1/2}(x) \subset B_{3/2}(y)$. From the disjointness of the $B_{1/2}(x_i)$ we then infer that any $y \in Q$ can lie in $B_1(x_i)$ for at most 
$$
\left\lfloor \frac{|B_{3/2}|}{|B_{1/2}|} \right\rfloor = 10
$$
values of $i$. We then have
\begin{equation} \label{eq-mqr}
m |B_1| \le 10 |Q| \le 10 |\Ne_1(\partial R)|;
\end{equation}
combining \eqref{eq-qmb} and \eqref{eq-mqr}, we obtain
\begin{equation}
|\partial Q| \le 10 \frac{|\partial B_1|}{|B_1|}  |\Ne_1(\partial R)| \le 25 |\Ne_1(\partial R)|,
\end{equation}
as required.
\end{proof}

\subsection{An isoperimetric inequality} Recall the  classical hyperbolic isoperimetric inequality  proved by Schmidt \cite{schmidt}.

\begin{theorem} \label{thm-explicit-iso}
For any $r > 0$, any region enclosed by a curve at length $2\pi \sinh r$ has area at most $2\pi (\cosh r - 1)$. 
\end{theorem}
These functions of $r$ are of course the perimeter and area of a hyperbolic disk of radius $r$. 
Since $\cosh r - \sinh r < 1$ for any $r > 0$ we immediately have
\begin{theorem}\label{thm-iso} 
If $Q\subset \Ha$ is a domain then 
$|Q|\le |\partial Q|$.
\end{theorem}

We can use this to prove the following
\begin{lemma} \label{lem-annular-iso}
Suppose $\beta_0$ and $\beta_1$ are two disjoint homotopic curves on a hyperbolic surface $S$ that together bound a compact annular region $A$.
Then 
\begin{equation} \label{eq-annular-iso}
|A| \le |\beta_0| + |\beta_1|.
\end{equation}
\end{lemma}
\begin{proof}
Draw a simple smooth arc $\gamma$ in $A$ connecting $\beta_0$ to $\beta_1$. 
Then $A \setminus \gamma$ is connected and simply connected,
and the universal cover of $A$ in $\Ha$ is an infinite concatenation of closures of $A \setminus \gamma$;
the area of $n$ of consecutive regions is $n|A|$, and the perimeter of this concatenation of $n$ consecutive regions is $2 |\gamma| + n(|\beta_0| + |\beta_1|)$. Letting $n \to \infty$ and applying Theorem \ref{thm-iso}, we obtain \eqref{eq-annular-iso}.
\end{proof}

We can now easily prove the fundamental estimate of this subsection.
\begin{proposition}\label{prop-iso} Let $S$ be a hyperbolic surface, 
and supposed we are given a closed geodesic $\alpha \subset S$, 
and a closed  smooth  simple curve $\beta\subset S$ homotopic to $\alpha$. 
Let $Q\subset S$ be the (possibly disconnected) domain  bounded by the two curves. 
Then $|Q|\le |\beta|$.
\end{proposition}

%

\begin{proof} First, we can approximate $\beta$ with simple smooth curves such that the intersection $\alpha\cap \beta$ is a finite set. In the rest of the proof we assume this to be the case.
We have two cases.

\noindent\textbf{Case 1.} Suppose  $\alpha$ and $\beta$ intersect each other. Then $Q$ is a disjoint union of simply connected domains $Q_1,...,Q_k$. Each $Q_i$ is bounded by segments $\alpha_i$ and $\beta_i$ of $\alpha$ and $\beta$ respectively,
and no two distinct $\beta_i$'s intersect.
We can double $Q_i$ over $\alpha_i$ to obtain
$2|Q_i| \le 2|\beta_i|$ and hence  
\begin{equation} \label{eq-qb}
|Q_i| \le |\beta_i|
\end{equation}
by Theorem \ref{thm-iso}.
Summing  \eqref{eq-qb} over all $i$, we obtain 
\begin{equation}\label{eq-q}
|Q| = \sum_{i=1}^{k} |Q_i| \le  \sum_{i=1}^k |\beta_i|  \le |\beta|.
\end{equation}

\noindent\textbf{Case 2.}  Suppose  $\alpha\cap \beta= \emptyset$. Then $Q\subset S$ is an embedded annulus. Doubling $Q$ over $\alpha$ and applying Lemma \ref{lem-annular-iso}, we obtain the desired result. 
%
%
%
\end{proof}

Suppose $R\subset S$ is a subsurface (with smooth boundary) of a closed hyperbolic surface $S$.  Let $\Omega \subset S$ be convex core of $U$. That is, $\Omega$ is the surface with geodesic boundary obtained by filling in the holes (contractible components of the complement) of $R$, and replacing each boundary curve of $R$ with its geodesic representative.

\begin{proposition}\label{prop-iso-1}  Suppose $R\subset S$ is subsurface (with smooth boundary) of a closed hyperbolic surface $S$, and let $\Omega$ be its convex core. Then
$$
 |\Omega|-|\partial R| \le |R|\le |\Omega|+|\partial R|.
$$
\end{proposition}
\begin{proof}
As above we may assume that $\partial{R} \cap \partial{\Omega}$ is a finite set. 
Denote by $\beta_1,...,\beta_m$ the connected components of $\partial R$. 
Suppose $\beta_i$ bounds a  simply connected domain $Q_i$. Then
$|Q_i|\le |\beta_i|$ by Theorem \ref{thm-iso}.  

Suppose now that $\beta_i$ is homotopic to a closed geodesic $\alpha_i$, and denote by $Q_i$ the union of all domains whose boundary is contained in $\alpha_i\cup \beta_i$. Then, $|Q_i|\le |\beta_i|$ by Proposition \ref{prop-iso}. Summing up these inequalities over $i=1,..,m$ proves the proposition.
\end{proof}

\section{The main lemmas and the proof of Theorem \ref{thm-main-geo}}\label{section-lemmas}

The proof of Theorem \ref{thm-main-geo} is built on two lemmas which we state next. We close the section with the proof of theorem assuming the lemmas.

\subsection{Genera of  geometrically random surfaces} 
The first lemma estimates from below the genus of a geometrically random surface.
Before stating the lemma we adopt the following definition.
\begin{definition}\label{definition-sc}
$$
\Su_\epsilon(T,C)=\{\Sigma\in \Su_\epsilon(T):  4\pi(\gen(\Sigma)-1) \ge (1+C\epsilon^2)\Area_\M(\Sigma)\}.
$$
\end{definition}

\begin{lemma}\label{lemma-comp} There exists a constant $C>0$ so that
$$
\lim_{T\to \infty} \frac{| \Su_\epsilon(T,C)|}{|\Su_\epsilon(T)|}=1
$$
when $\epsilon$ is  sufficiently small.
\end{lemma}
\begin{remark} If $\Sigma$ is a totally geodesic surface of area $T$ then  $4\pi(\gen(\Sigma)-1) =T$. Therefore, the previous lemma shows that the genus of a geometrically random surface is significantly larger than the genus of a totally geodesic surface of the same area. This is the key factor in the proof of  Theorem \ref{thm-main-geo}.
\end{remark}

To put this into context, we refer to Figure \ref{Fig-1}
which (schematically) shows the distribution of minimal surfaces by their area and genus. The key features of this plot are as
follows:
\begin{itemize}
\item All surfaces lie on or above the $\gen - 1 = \frac{T}{4\pi}$ line
(which is the locus of the totally geodesic surfaces of genus $\gen$). This immediately
follows from the Gauss-Bonnet formula (and from the formula
 for the Gaussian curvature of a minimal surface).
\item All surfaces lie below the $\gen - 1 = (1 +
C_1^2\epsilon^2)\frac{T}{4\pi}$ line. This follows from Seppi's
curvature estimate (Theorem~\ref{thm-seppi}), where $C_1$ is the
constant from that theorem.
\item Kahn--Marković's upper bound \cite{k-m-1} tells us that the count
density of surfaces increases superexponentially with their genus (but
depends much more weakly on their area); moreover Müller--Puchta's lower
bound \cite{m-p} tells us that this superexponential growth can be
observed even among the covers of a single surface (which have constant
ratio of $\gen - 1$ to area).
\item From this it is not hard to deduce that the subset
$\Su_\epsilon(T,C_2)$ (where $C_2$ is the constant of
Lemma~\ref{lemma-comp}) dominates the count of $\Su_\epsilon$, as long
as it is nonempty. Verifying this latter condition
(Proposition~\ref{prop-model-1}) is the most substantial part of the proof.
\end{itemize}

\begin{figure}[t]
\begin{center}
\includegraphics[height=5cm]{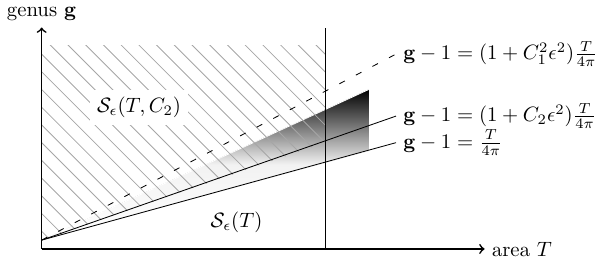}
\caption{}
\label{Fig-1}
\end{center}
\end{figure}

\subsection{Nearly geodesic minimal surfaces and the totally geodesic locus}
The second key lemma bounds from above the genus of an $\epsilon$-nearly geodesic minimal surface which  concentrates too much on the totally geodesic locus $\HH$.  Recall the constant $\delta_1(h)$ (see the remark after Proposition \ref{prop-prva}).

\begin{lemma}\label{lemma-comp-1} There exist  constants $C>0$, and $\epsilon_0(h)>0$,  such  that for every $\Sigma\in \Su_\epsilon(T)$, the inequality
\begin{equation}\label{eq-jutro}
4\pi(\gen(S)-1)\le \left(1+   C\left(1-\frac{|S(\delta_1(h),h))|}{|S|}\right)\epsilon^2 \right)T
\end{equation}
holds for every $h\in \N$ assuming $\epsilon<\epsilon_0(h)$. Here $f\from S\to \M$ is the $\epsilon$-nearly geodesic minimal surface in the class $\Sigma$.
\end{lemma}

\subsection{Proof of Theorem \ref{thm-main-geo}}

Let $\mu^\Geo _\epsilon$ be an $\epsilon$-geometrical limiting measure on $\G_2(\M)$. That is, there exists a sequence $T_n\to \infty$ such that 
$$
\mu^\Geo_\epsilon=\lim\limits_{T_{n}\to \infty} \frac{1}{|\Su_\epsilon(T_n)|} \sum_{\Sigma\in \Su_\epsilon(T_n)} \mu(\Sigma)
$$
(recall that $\mu(\Sigma)$ is the push forward of the hyperbolic area measure from $S$  to $\G_2(\M)$).
Let  $C_1$ be the constant from Lemma \ref{lemma-comp}. Then
\begin{equation}\label{eq-vot}
\mu^\Geo_\epsilon=\lim\limits_{T_{n}\to \infty} \frac{1}{|\Su_\epsilon(T_n,C_1)|} \sum_{\Sigma\in \Su_\epsilon(T_n,C_1)} \mu(\Sigma).
\end{equation}

Let $\Sigma\in \Su_\epsilon(T_n,C_1)$ and denote by $f\from S\to \M$ the minimal surface representing $\Sigma$. Let $h\in \N$, and denote by $C_2$ and $\epsilon_0(h)$ the constants from Lemma \ref{lemma-comp-1}. Applying Lemma \ref{lemma-comp-1} to $\Sigma$ yields the inequality  
\begin{equation}\label{eq-vot-1}
\frac{|S(\delta_1(h),h))|}{|S|}\le 1-\frac{C_1}{C_2}<1
\end{equation}
assuming $\epsilon<\epsilon_0(h)$.

In turn, the inequality (\ref{eq-vot-1}) yields the estimate
$$
\mu(\Sigma)\big(\Ne_{\delta_{1}(h)}(\HH_h)\big) \le \frac{|S(\delta_1(h),h)|}{|S|}< 1-\frac{C_1}{C_2}
$$
for every $\Sigma\in  \Su_\epsilon(T_n,C)$ assuming $\epsilon<\epsilon_0(h)$. Combining this with (\ref{eq-vot}) gives  the estimate
$$
\mu^\Geo_\epsilon\big(\Ne_{\delta_{1}(h)}(\HH_h)\big)\le 1-\frac{C_1}{C_2}
$$
for every $\epsilon$-geometrical limiting measure on $\G_2(\M)$ when $\epsilon<\epsilon_0(h)$. In turn, this shows that
\begin{equation}\label{eq-vot-2}
\mu^\Geo\big(\Ne_{\delta_{1}(h)}(\HH_h)\big)\le1-\frac{C_1}{C_2}
\end{equation}
for every geometrical limiting measure $\mu^\Geo$ on $\G_2(\M)$. We observe that (\ref{eq-vot-2})  holds for every $h\in \N$, and that the constant on the right hand side does not depend on $h$. Define $q=\frac{C_1}{C_2}$, and  note that $q$ only depends on $\M$. Then  (\ref{eq-vot-2}) implies  
$$
\mu^\Geo(\HH_h)\le1-q.
$$
Since $\HH_h$ is an increasing sequence of sets, we have
$$
\mu^\Geo(\HH)\le \lim_{h\to \infty}  \mu^\Geo(\HH_h) \le 1-q.
$$
Thus, $q\le |\mu^\Geo_{\mathcal{L}}|$, and the proof is complete.

\section{Genera of  geometrically random surfaces I}\label{section-4}

The proof of Lemma \ref{lemma-comp} occupies Sections~\ref{section-4} to \ref{section-model}.
The  proof is carried out in Section~\ref{section-5}, by combining two
contrasting statements:
\begin{itemize}
\item Proposition~\ref{prop-party-1}, which gives an upper bound on the
rate of growth of the complement of $\Su_\epsilon(T,C)$, and easily
follows from Kahn--Markovi\'c's upper bound~\cite{k-m-1};
\item and Proposition~\ref{prop-party-2}, which gives a lower bound on
the rate of growth of $\Su_\epsilon(T,C)$ itself. We obtain this bound
in Section~\ref{section-6}, just by counting the covers of a single
element of $\Su_\epsilon(T_0,C)$ (using M\"uller--Puchta's lower bound
\cite{m-p}).
\end{itemize}
However, this last argument relies on $\Su_\epsilon(T,C)$ being
nonempty: 
\begin{proposition}\label{prop-model-1}  There exists a universal constant $C>0$ such that for each small enough $\epsilon$ there exists $T=T(\epsilon)$ so that $\Su_\epsilon(T,C)\ne \emptyset$.
\end{proposition}
The goal of this section is to prove Proposition \ref{prop-model-1}. To do so,  we need to  construct a
surface satisfying the defining inequality (\ref{eq-l-1}) of
$\Su_\epsilon(T,C)$. This is done in two steps:
\begin{itemize}
\item First, we use a recent result of Rao \cite{rao} to reduce it to the simpler problem of
finding a ``model'' surface that satisfies (a stronger version of) that
inequality, but which exists independently of $\M$.
\item The existence of such a ``model surface'' is the content of
Proposition~\ref{prop-model}. We construct it in Section~\ref{section-model},
essentially by hand.
\end{itemize}

\subsection{The model quasifuchsian manifold}
Recall that a quasifuchsian 3-manifold is $K$-quasifuchsian  if its limit set is a $K$-quasicircle.  The following proposition is proved in Section \ref{section-model}.

\begin{proposition}\label{prop-model} There exists a universal constant $C>0$ with the following property. For every small enough $\epsilon>0$  there exists a 
$(1+\epsilon)$-quasifuchsian manifold $N$ of genus two such that  
\begin{equation}\label{eq-t-0}
T\le \frac{4\pi}{1+C\epsilon^2},
\end{equation}
where $T$ is the area of  the minimal surface homotopic to $N$.
\end{proposition}

\subsection{Rao's theorem} In \cite{rao} Rao proved a version of the Ehrenpreis conjecture for quasifuchsian manifolds of genus two. We state it in a form that is convenient for our purposes.
\begin{theorem}\label{thm-rao} Let $\M$ be a closed 3-manifold, and let $\eta>0$. Then given any quasifuchsian 3-manifold $N_0$ of genus two, we can find:
\begin{itemize} 
\item a finite cover $N\to N_0$, 
\item  a map $F\from N\to \M$ which is locally $(1+\eta)$-bilipschitz.
\end{itemize}
\end{theorem}

\subsection{Proof of  Proposition \ref{prop-model-1}}

Let $N_0$ be the quasifuchsian manifold, and $C_0$ the  constant,  from Proposition \ref{prop-model}. We choose  $N_0$ so it is
a $(1+\frac{\epsilon}{2})$-quasifuchsian manifold. Let $g_0\from R_0\to N_0$ be the minimal surface homotopy equivalent to $N_0$ (here $R_0$ is a Riemann surface of genus two). Then from (\ref{eq-t-0}) we get
\begin{equation}\label{eq-t-00}
\Area_\M(g_0(R_0))\le \frac{4\pi}{1+C_0\epsilon^2}.
\end{equation}

Set $\eta=\epsilon^3$. We apply Theorem   \ref{thm-rao}  and get the corresponding finite cover $N\to N_0$, and the $(1+\eta)$-bilipschitz map $F\from N\to \M$. Let $g\from R\to N$ be the lift of the map $g_0\from R_0\to N_0$, where $R$ is the corresponding cover $R\to R_0$. Then $g\from R\to N$  is the minimal surface homotopy equivalent to $N$, and from (\ref{eq-t-00}) we get
\begin{equation}\label{eq-t-1p}
\Area_\M(g(R))\le \frac{4\pi(\gen(R)-1)}{1+C_0\epsilon^2}.
\end{equation}

On the other hand, consider the composition $F\circ g\from R \to \M$. Note that the quasifuchsian group $(F\circ g)_*\from \pi_1(R)\to \pi_1(\M)$ is $(1+\epsilon)$-quasifuchsian when $\eta$ is small enough.  Moreover, since $F$ is $(1+\eta)$-bilipschitz we have
$$
\Area_\M\big( (F\circ g)(R)\big)\le (1+\eta)^2\Area(g(R)).
$$
Combining this with (\ref{eq-t-1p}) we get
\begin{equation}\label{eq-prj}
\Area_\M\big( (F\circ g)(R)\big)\le (1+\eta)^2 \frac{4\pi(\gen(R)-1)}{1+C_0\epsilon^2}.
\end{equation}
Since $\eta=\epsilon^3$ we get 
$$
\frac{(1+\eta)^2}{1+C_0\epsilon^2}=\frac{(1+\epsilon^3)^2}{1+C_0\epsilon^2}\le \frac{1}{1+C_1\epsilon^2}
$$
when $\epsilon$ is small enough and $C_1=\frac{C_0}{2}$. Replacing this into (\ref{eq-prj}) gives
$$
\Area_\M\big( (F\circ g)(R)\big)\le  \frac{4\pi(\gen(R)-1)}{1+C_1\epsilon^2}.
$$
Let $\Sigma$ be the homotopy class of $F\circ g\from R \to \M$.  Since $\Area_\M(\Sigma) \le \Area_\M\big( (F\circ g)(R)\big)$, we get
$$
\Area_\M(\Sigma) \le  \frac{4\pi(\gen(\Sigma)-1)}{1+C_1\epsilon^2}.
$$
Thus, $\Sigma\in \Su_\epsilon(T,C_1)$, and the proposition is proved.

\section{Genera of  geometrically random surfaces II}\label{section-5}

In this section we prove Lemma \ref{lemma-comp} assuming the following two propositions which we prove in the next  section.

\begin{proposition}\label{prop-party-1} There exists a constant $\alpha>0$ such that for every $C>0$ the inequality
$$
|\Su_\epsilon(T)|-|\Su_\epsilon(T,C)| \le  (\alpha r)^{2r}
$$
holds, where 
$$
r=r(T,C)=\frac{(1+C\epsilon^2)T}{4\pi}+2.
$$
\end{proposition}

\begin{proposition}\label{prop-party-2} There exist constants $\beta , C_0>0$  such that for every $\epsilon$ small enough there exists $T_0=T_0(\epsilon)$ so that  the inequality
$$
(\beta s)^{2s} \le     |\Su_\epsilon(T,C_0)| 
$$
holds for every large enough $T$. Here
$$
s=s(T,C_0)=\frac{(1+C_0\epsilon^2)(T-T_0)}{4\pi}.
$$
\end{proposition}

\subsection{Proof of Lemma \ref{lemma-comp}}
Let $0<C_1<C_0$, where $C_0$ is the constant from Proposition \ref{prop-party-2}.
In the remainder of the proof we show
\begin{equation}\label{eq-rib}
\lim_{T\to \infty} \frac{|\Su_\epsilon\big(T,C_1\big)|}{|\Su_\epsilon(T)|}=1
\end{equation}
which proves the lemma.

We have
\begin{equation}\label{eq-pesma}
\frac{|\Su_\epsilon(T)|-|\Su_\epsilon(T,C_1)|}{|\Su_\epsilon(T)|}\le \frac{(\alpha r)^{2r}}{(\beta s)^{2s}}
\end{equation}
where we used the above propositions to estimate the numerator and denominator respectively. Here $r=r(T,C_1)$, and $s=s(T,C_0)$. Replacing the values of $r(T,C_1)$ and $s(T,C_0)$ from the statements of the propositions, we get
$$
\lim_{T\to \infty} \frac{r(T,C_1)}{s(T,C_0)}=\frac{(1+C_1\epsilon^2)}{(1+C_0\epsilon^2)}<1
$$
because $C_1<C_0$. But this implies that the right hand side of the inequality (\ref{eq-pesma}) tends to $0$ when $T\to \infty$ which proves (\ref{eq-rib}).

\section{Genera of  geometrically random surfaces III}\label{section-6}

We prove the two propositions from the previous section.

\subsection{Proof of Proposition \ref{prop-party-1}}

The proof is based on the counting result by Kahn-Markovi\'c \cite{k-m-1} which states that there exists $\alpha>0$ such that the number of homotopy classes of all surfaces of genus at most $g$ in $\M$ is smaller than  $(\alpha g)^{2g}$. In particular this enables us to bound the number of $\epsilon$-quasifuchsian surfaces of genus at most $g$. That is, we have
\begin{equation}\label{eq-sin}
|\Su_\epsilon(g)|\le (\alpha g)^{2g},
\end{equation}
where we recall that $\Su_\epsilon(g)$ is the set of classes in $\Su_\epsilon$ of genus at most $g$.

Now, from the definition of the set $\Su_\epsilon(T,C)$  we derive the inclusion 
$$
\left(\Su_\epsilon(T)\setminus \Su_\epsilon(T,C)\right)\subset \Su_\epsilon(g)
$$ 
where 
$$
g=\left\lceil\frac{(1+C\epsilon^2)T}{4\pi} \right\rceil+1 \le \frac{(1+C\epsilon^2)T}{4\pi}+2= r(T,C).
$$
Combining this with (\ref{eq-sin}) yields
$$
|\Su_\epsilon(T)|-|\Su_\epsilon(T,C)| \le |\Su_\epsilon(g)|\le (\alpha g)^{2g}\le (\alpha r)^{2r} ,
$$
which proves  the proposition.

\subsection{Proof of Proposition \ref{prop-party-2}}
The proof of this proposition is based on Proposition \ref{prop-model-1} and the counting result by M\"uller-Puchta \cite{m-p} (see also \cite{k-m-1}) that the number of different degree $n$ covers of a closed hyperbolic surface is at least $(\beta g_n)^{2g_{n}}$, for every $n$ large enough, and  for some $\beta>0$ depending on the surface. Here $g_n$ denotes the genus of the  covering surface of
degree $n$.

Fix $\epsilon$ which is small enough so that Proposition \ref{prop-model-1} holds. Let $\Sigma_0\in \Su_\epsilon(T_0)$, where $T_0=T(\epsilon)$, denote the class from Proposition \ref{prop-model-1}.  From the definition of  $\Su_\epsilon(T_0,C_0)$ we get
\begin{equation}\label{eq-t-mars}
\frac{(1+C_0\epsilon^2)T_0}{4\pi}+1\le \gen(\Sigma_0)
\end{equation}
where $C_0$ is the constant from  Proposition \ref{prop-model-1}.

Let $\Sigma$ be a degree $n$ cover of $\Sigma_0$. Then $\Sigma\in \Su_\epsilon(nT_0,C_0)$, and the genus of $\Sigma$ is given by
$\gen(\Sigma)=n(\gen(\Sigma_0)-1)+1$. 
Therefore by  M\"uller-Puchta, for large $n$  we have
\begin{equation}\label{eq-mp}
(\beta g_n)^{2g_{n}} \le |\Su_\epsilon(nT_0,C_0)|
\end{equation}
where $g_n=n(\gen(\Sigma_0)-1)+1$. 

Let $T$ be large, and $n$ an integer such that $nT_0\le T \le (n+1)T_0$. Then from (\ref{eq-mp}) we get
\begin{equation}\label{eq-mp-1}
(\beta g_n)^{2g_{n}} \le |\Su_\epsilon(nT_0,C_0)|\le |\Su_\epsilon(T,C_0)|.
\end{equation}
From (\ref{eq-t-mars}) we find 
$$
\frac{(1+C_0\epsilon^2)nT_0}{4\pi}\le n\big(\gen(\Sigma_0)-1\big)=g_n-1,
$$
implying that 
$$
s(T,C_0)=\frac{(1+C_0\epsilon^2)(T-T_0)}{4\pi} \le \frac{(1+C_0\epsilon^2)nT_0}{4\pi}\le g_n
$$
where we used the fact that $T-T_0\le nT_0$. Replacing this into (\ref{eq-mp-1}) proves the proposition.

\section{The model quasifuchsian manifold of genus two}\label{section-model}

In this section we construct the model quasifuchsian manifold from Proposition \ref{prop-model}. 

\subsection{Complex earthquakes} Let $R$ denote a hyperbolic Riemann surface and suppose $(\Lambda,\mu)$ is a measured lamination. This means that $\Lambda$ is a lamination on $R$, and $\mu$ a complex valued transverse measure. Let $G<\text{Isom}(\Ha)$ be the Fuchsian group such that  $R\approx \Ha/G$. Recall  the complex earthquake  developing map $E_{(\Lambda,\mu)}\from \Ha\to {\Ho}$. This map is equivariant with respect to $G$, and it induces a homomorphism $e\from G\to \text{Isom}({\Ho})$.

The norm $||\mu||$ is the supremum of the total measure of $|\mu|$ over all transverse intervals of length one. Assuming  $||\mu||$ is small enough, it follows that the homomorphism $e\from G\to \text{Isom}({\Ho})$ is quasifuchsian, that is the image group $e(G)$ is quasifuchsian. The following quantitative version of this claim is well known (for example, see Lemma 4.14 in \cite{e-m-m}).
\begin{theorem}\label{thm-emm} There exist  universal constants $C,a>0$ such that the group $e(G)$ is 
$(1+C||\mu||)$-quasifuchsian when $||\mu||\le a$. Here $R$ is any hyperbolic surface, and $(\Lambda,\mu)$ any measured lamination on $R$.
\end{theorem}

\subsection{Bending along a closed geodesic} Now fix some
closed Riemann surface of genus two. We name it $R$. Fix a simple closed geodesic $\gamma\subset R$, and let $\Lambda$ be the lamination whose only leaf  is $\gamma$. Given $t>0$, we let $\mu_t$ be the transverse measure on $\Lambda$ whose weight on $\gamma$ is equal to $t\IM$ 
(here $\IM$ is the imaginary unit). By $e_t\from G\to \text{Isom}({\Ho})$ we denote the corresponding homomorphism induced by the measured lamination $(\Lambda,\mu_t)$.

Let $D_1=||\mu_1||$. Then $||\mu_t||=tD_1$. From Theorem \ref{thm-emm} we get that when $t$ is small enough, the group $e_t(G)$ is a $(1+C_1D_1t)$-quasifuchsian group of genus two (here $C_1$ denotes the constant from Theorem \ref{thm-emm}). Let $N_t$ denote the quasifuchsian manifold corresponding to $e_t(G)$. We have proved the following claim.
\begin{claim}\label{claim-aero} There exists a constant $C>0$ such that the quasifuchsian manifold  $N_t$ is a 
$(1+Ct)$-quasifuchsian manifold. 
\end{claim}

It remains to estimate the area of the minimal surface homotopic to $N_t$.

\subsection{The area of the minimal surface in $N_t$} To prove Proposition \ref{prop-model} we need to find a suitable upper bound on the area of the minimal surface in $N_t$. We obtain such a bound by producing an explicit surface (not necessarily minimal) homotopic to $N_t$, and estimating its area from above. 

One obvious surface homotopic to $N_t$ is the pleated surface $R_t\subset N_t$. Let $E_t\from R\to N_t$ be the induced pleating map. The map $E_t$ is locally an isometry, except at $\gamma$ where it bends the surface $R$ by the angle $t$. Set $E_t(R)=R_t$. Then the area of $R_t$ is given by 
\begin{equation}\label{eq-t-f1}
\Area_{N_{t}}(R_t)=4\pi.
\end{equation}
Therefore, to prove the lemma we need to find a surface of smaller area than $R_t$. We do that by bevelling the surface $R_t$ near the geodesic $E_t(\gamma)\subset N_t$.

Fix $\alpha,\beta\subset R$ which are equidistant, and embedded, lines either side of the geodesic $\gamma$. By $U\subset R$ we denote the  annulus (embedded in $R$) bounded by $\alpha$ and $\beta$. The annulus $U$ is foliated by geodesic segments $\{s_x\}_{x\in \gamma}$,  where each $s_x$ intersects $\gamma$ orthogonally at $x$,  and the endpoints of $s_x$ are on $\alpha$ and $\beta$ respectively.

\begin{figure}[t]
\begin{center}
\includegraphics[height=5cm]{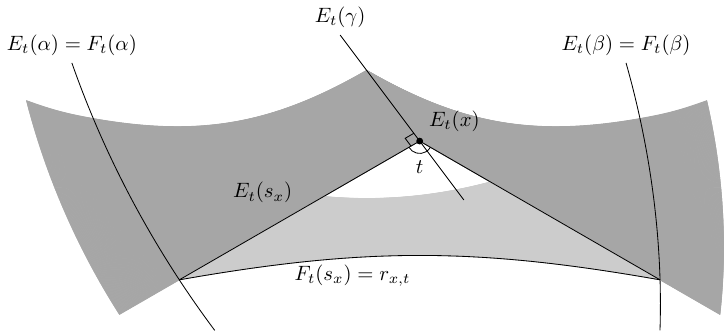}
\caption{The hyperbolic bevel $F_t(U)$ associated to the hyperbolic dihedral angle $E_t(U)$.}
\label{Fig-2}
\end{center}
\end{figure}

Define a new map $F_t\from R\to N_t$ as follows. We let  $F_t=E_t$ away from the annulus $U$. On $U$, we define $F_t$ so it maps the segment $s_x$ onto the geodesic segment $r_{x,t}$ which has the same endpoints as the piecewise geodesic arc $E_t(s_x)$. Let $S_t=F_t(R)$. We show that the area of $S_t$ admits the required upper bound.

Let $V_t=F_t(U)$. Then $V_t$ is a hyperbolic bevel which is placed to make a triangular shape together with the sloping surface $E_t(U)$ (see Figure \ref{Fig-2}).  Let $d_0=\frac{|s_x|}{2}$, and $d_t=\frac{|r_{x,t}|}{2}$ (obviously $d_0$ and $d_t$ do not depend on the choice of $x\in \gamma$). From the hyperbolic law of sines we compute
\begin{equation}\label{eq-hyp}
\sinh(d_t)=\cos\left(\frac{t}{2}\right)\sinh(d_0).
\end{equation}
Moreover, using elementary hyperbolic geometry we compute the hyperbolic areas of $U$ and $V_t$, and get
\begin{equation}\label{eq-hyp-1}
\Area_{\M}\big(E_t(U)\big)=2|\gamma|\sinh(d_0),\,\,\,\,\,\,\,\,\,\,   \Area_{N_{t}}(V_t)=2|\gamma|\tanh(d_t) \cosh(d_0).
\end{equation}

We have
\begin{align*}
\Area_{N_{t}}(S_t)&=\Area_{N_{t}}(R_t)-(\Area_{\M}\big(E_t(U)\big)-\Area_{N_{t}}(V_t)) \\
&=4\pi -2|\gamma|\big(\sinh(d_0)-\tanh(d_t) \cosh(d_0)\big)\\
&= 4\pi\left( 1- C_1 \big(1-\tanh(d_t) \coth(d_0)\big)   \right)
\end{align*}
where $C_1=\frac{2|\gamma|\sinh(d_0)}{4\pi} $. Then 
$$
\Area_{N_{t}}(S_t) \le \frac{4\pi}{1+C_1 \big(1-\tanh(d_t) \coth(d_0)\big)}.
$$
But from (\ref{eq-hyp}) we derive $d_t\le d_0-C_2t^2$ for some $C_2>0$, which gives 
$C_3t^2 \le (1-\tanh(d_t) \coth(d_0))$ for some $C_3>0$. Thus, 
$$
\Area_{N_{t}}(S_t) \le \frac{4\pi}{1+C_1C_3t^2}.
$$
Therefore we have proved the following claim (note that the minimal  surface homotopy equivalent to $N_t$ minimises the area in its homotopy class). 
\begin{claim}\label{claim-aero-1} There exists a constant $C>0$ such that if $T_t$ denotes the area of the minimal surface homotopic to $N_t$, then for every small enough $t$ we have 
$$
T_t \le \frac{4\pi}{1+Ct^2}.
$$
\end{claim}
Combining this with Claim \ref{claim-aero} proves Proposition \ref{prop-model}.

\section{Local and global properties of minimal maps}

The proof of Lemma \ref{lemma-comp-1} occupies Sections 8 to 11. We roughly outline the idea of the proof.  As usual, $f\from S\to \M$ denotes an $\epsilon$-nearly geodesic minimal surface. The  application of the Gauss-Bonnet formula yields
$$
4\pi(\gen(S)-1)=\int\limits_{S} (1+\la^2(x))\,dA_\sigma \le \left(1+\frac{2}{|S|}\int\limits_{S} \la^2(x)\,dA \right)T,
$$
where $T$ is the area of  the minimal surface $f(S)$, and $\lambda(x)$ the principal curvature at the point $f(x)$. Thus, to prove the lemma  it suffices to prove the bound
\begin{equation}\label{eq-out}
\frac{2}{|S|}\int\limits_{S} \la^2(x)\,dA\le C\left(1-\frac{|S(\delta_1(h),h))|}{|S|}\right)\epsilon^2.
\end{equation}
The two main ingredients in proving this bound are as follows:
\begin{itemize}
\item Seppi's universality estimate $|\lambda(x)|\le C_1\epsilon$ (see Theorem \ref{thm-seppi} below).
\item Each connected component  of $S(\delta_1(h),h)$ has a rigid convex core $\Omega$  (see Definition \ref{def-rigid} below). This implies that for $x\in \text{int}(\Omega)$ the relation $|\lambda(x)|= o(1) \epsilon$ holds, where $o(1)\to 0$  when $\dis_S(x,\partial \Omega)\to \infty$.
\end{itemize}
The purpose of the next three sections is to prove Lemma \ref{lemma-close} which estimates the integral of $\la^2$ over rigid subsurfaces. 
The proof of Lemma \ref{lemma-comp-1} is completed in Section \ref{section-comp-1}. 
 
\subsection{Global bound on principal curvatures} 
Let $f\from S\to \M$ be an $\epsilon$-nearly geodesic minimal surface. By $\wt{f}\from \Ha\to {\Ho}$ we denote its lift  (here we identify  $\Ha$ and ${\Ho}$ with the universal covers of $S$ and $\M$ respectively). Clearly $\wt{f}$ is a quasiisometry (because it is a developing map of a quasifuchsian representation). But we need a more quantitative version of this which we derive first. Then we recall the basic properties of the distance function between $\wt{S}$ and any totally geodesic plane $P\subset {\Ho}$. These results are used in the next two sections where  we obtain an upper bound on  the principal curvature of $S$ on its subsurfaces which are rigid.

Let $\la\from S\to [0,\infty)$ be the non-negative eigenvalue of the shape operator corresponding to the minimal map $f\from S\to \M$. The significance of the function $\la$ is that the principal curvatures of a minimal surface $f(S)$ at the point $f(x)$ are equal to $\la(x)$ and $-\la(x)$ respectively.  The following global estimate was shown by Seppi  \cite{seppi}.
\begin{theorem}\label{thm-seppi} There exists a universal constant $C$ such that 
\begin{equation}\label{eq-seppi}
||\lambda||_{\infty}\le C \epsilon
\end{equation}
for every $\epsilon$-nearly geodesic minimal surface $f\from S\to \M$ when  $\epsilon$ is small enough.
\end{theorem}
The first corollary we derive from this theorem concerns the quasigeodesic properties of the minimal map $\wt{f}$. 
\begin{proposition}\label{prop-rata} There exists a constant $C>0$ such that if
$\alpha\subset \Ha$ is any geodesic, and $\beta\subset{\Ho}$ the geodesic with the same endpoints as the quasigeodesic $\wt{f}(\alpha)$, then 
$\wt{f}(\alpha)\subset \Ne_{C\epsilon}(\beta)$.
\end{proposition}
\begin{proof} Let $x\in \alpha$. Then the infinitesimal bending of the curve $\wt{f}(\alpha)$ at the point $\wt{f}(x)$ is bounded above by the maximum of principal curvatures at $\wt{f}(x)$ which is equal to $\la(x)$. Let $C_1$ denote the constant from Theorem \ref{thm-seppi}. 
Then the  infinitesimal bending of the curve $\wt{f}(\alpha)$ at all of its points is bounded above by $C_1\epsilon$.  

Let $g\from \alpha\to \wt{f}(\alpha)$ be the path-length reparametrization of the quasigeodesic $\wt{f}(\alpha)$. From  Lemma 4.4 in \cite{e-m-m} we conclude that 
$g(\alpha)\subset \Ne_{C_{2}\epsilon}(\beta)$  which proves the proposition.
\end{proof}

\subsection{The induced metric on $S$} Denote by  $\sigma^2(z)|dz|^2$  the conformal metric on $S$ obtained by pulling back the hyperbolic metric from $\M$ by the minimal map $f$ (here $z$ is a local complex parameter on $S$). The Gauss curvature of this metric $\sigma$ is equal to $-(1+\la^2)$. Let $\rho^2(z)|dz|^2$  denote the density of the hyperbolic metric on $S$.  Recall the Ahlfors estimate:
\begin{equation}\label{eq-ahlfors}
\frac{\rho^2(z)}{1+\la^2(x)} \le \sigma^2(x)\le \rho^2(z).
\end{equation}
The following is an immediate corollary of  (\ref{eq-ahlfors}).
\begin{proposition}\label{prop-ahlfors} Let $\phi\from \Ha\to \R$ be any function, and let $\nabla_\sigma \phi$ and $\Delta_\sigma \phi$ denote respectively the gradient and the laplacian of $\phi$ with respect to the lift of the metric $\sigma$ to $\Ha$. Then 
$$
|\nabla \phi|\le |\nabla_\sigma \phi| \le 2 |\nabla \phi|,
$$
and 
$$
|\Delta \phi|\le |\Delta_\sigma \phi| \le 2 |\Delta \phi|,
$$
assuming $\epsilon$ is small enough. 
\end{proposition}
\begin{remark}
Here, and in the rest of the paper,  $\nabla \phi$, $\Delta \phi$ denote  the gradient and the laplacian of $\phi$ with respect to the hyperbolic metric.
\end{remark}
\subsection{The distance function}
Seppi also obtained local estimates on $\la$ which we utilise below.  Let $P\subset {\Ho}$ be a geodesic plane and define the function $u\from \Ha\to \R$ by 
\begin{equation}\label{eq-u}
u(x)=\sinh(\dis_{\M}(\wt{f}(x),P)).
\end{equation}
Here $\dis_{\M}(\wt{f}(x),P)$ is the signed distance from the point $\wt{f}(x)$ to the plane $P$. 
A significant property of minimal surfaces is that the distance function $u$ satisfies the second order elliptic PDE
\begin{equation}\label{eq-lap}
\Delta_\sigma u=2u
\end{equation}
on $\Ha$ (for example see \cite{seppi}).

For a function $\phi\from \Ha\to \R$, we let 
$$
||\phi||_{C^0(B_{r}(x))}=\max\{|\phi(y)|: y\in B_r(x)\}.
$$
The following local estimates were established in \cite{seppi} by applying the standard Schauder theory to  the formula (\ref{eq-lap}). 
\begin{proposition}\label{prop-esta} There exists a constant $C>0$ such that for every $x\in \Ha$ the inequalities
$$
|\nabla u(x)|  \le  C ||u||_{C^0(B_{1}(x))},
$$
and
$$
\la(x)  \le  \frac{C ||u||_{C^0(B_{1}(x))}}{\sqrt{ 1-C ||u^2||_{C^0(B_{1}(x))}} },
$$
hold when $\epsilon$ is small enough.
\end{proposition}
\begin{remark} Cleary, the second estimate makes sense only when $||u||_{C^0(B_{1}(x))}$ is small enough so that 
$C ||u^2||_{C^0(B_{1}(x))}<1$.
\end{remark}
\begin{proof} The first estimate is a special case of the formula (27) in \cite{seppi}, while the second one is the estimate (32) from \cite{seppi}. 
\end{proof}

\section{An auxiliary function}

The purpose of this section is to establish some basic properties of the auxiliary function $v\from \Ha \to [0,\infty)$ defined as the averaging of the square of the distance function $u$ from the previous section. The auxiliary function $v$ is used  in the proof of Lemma \ref{lemma-close} in the next section.

\subsection{The square of the distance function}

With a view towards applications that follow, we consider the square $u^2$. The advantage of considering the function $u^2$ is that it is non-negative and subharmonic.
\begin{proposition}\label{prop-lap-1} The inequality
\begin{equation}\label{eq-lap-1}
0\le u^2 \le \Delta u^2
\end{equation}
holds on $\Ha$.
\end{proposition}
\begin{proof}
Computing the laplacian of $u^2$ in terms of the function $u$ yields the inequality 
$$
 2u\Delta_\sigma u \le \Delta_\sigma u^2,
$$
which together with (\ref{eq-lap}) yields the estimate
$$
4u^2\le\Delta_\sigma u^2.
$$
Applying Proposition \ref{prop-ahlfors} we get $\Delta_\sigma u^2 \le 2\Delta u^2$ when $\epsilon$ is small enough. Replacing this in the previous displayed inequality  proves that $2u^2 \le \Delta u^2$ which in turn  proves the proposition.
\end{proof}

The fact that $u^2$ is subharmonic enables us to control the (local) supremum norm of $u^2$ in terms of its (local) $L^1$ norm.
\begin{proposition}\label{prop-svidj} For every $x\in \Ha$, the inequality
\begin{equation}\label{eq-svidj}
||u^2||_{C^0(B_{1}(x))} \le \int\limits_{B_{2}(x)} \, u^2(y)\, dA
\end{equation}
holds, where $dA$ is the hyperbolic area form.
\end{proposition}
\begin{proof} 
Let $y\in B_1(x)$. Since $\Delta u^2\ge 0$, the Mean Value Theorem yields the estimate
$$
u^2(y)\le \frac{1}{|B_{1}(y)|}\int\limits_{B_{1}(y)} \, u^2(z) \, dA(z).
$$
From $|B_{1}(y)|>1$, we get
$$
\frac{1}{|B_{1}(y)|}\int\limits_{B_{1}(y)} \, u^2(z) \, dA(z)\le \int\limits_{B_{1}(y)} \, u^2(z) \, dA(z)\le \int\limits_{B_{2}(x)} \, u^2(z) \, dA(z),
$$
where in the last inequality we used the fact that $B_1(y)\subset B_2(x)$ for every  $y\in B_1(x)$.  This together with the mean value estimate proves the proposition.

\end{proof}

On the other hand, computing  $\nabla u^2$ in terms of $u$, and applying Proposition \ref{prop-esta}, yields the following proposition.
\begin{proposition}\label{prop-laba} There exist constants $C,\eta>0$ such that assuming $\epsilon$ is small enough,
for every $x\in \Ha$ the inequality
\begin{equation}\label{eq-laba}
|\nabla u^2(x)|  \le  C||u^2||_{C^0(B_{1}(x))}
\end{equation}
holds. Moreover, if $||u^2||_{C^0(B_{1}(x))}<\eta$, then
\begin{equation}\label{eq-laba'}
\la^2(x)  \le  C ||u^2||_{C^0(B_{1}(x))}.
\end{equation}
\end{proposition}
\begin{proof} Let $C_1$ be the constant from Proposition \ref{prop-esta}. Since $\nabla u^2=2u\nabla u$, it follows that 
(\ref{eq-laba}) holds for $C=2C_1^2$.
On the other hand, if $||u^2||_{C^0(B_{1}(x))}<\frac{1}{4C^2_{1}}$, then (\ref{eq-laba'}) follows from the second inequality in Proposition \ref{prop-esta} by letting $C=2C_1^2$. Thus, the proposition holds for $C=2C_1^2$, and $\eta=\frac{1}{4C^2_{1}}$.

\end{proof}

\subsection{The averaging function and its properties}
For $x\in \Ha$, we set
$$
v(x)=\int\limits_{B_{2}(x)} \, u^2(y)\, dA(y).
$$
The following propositions list important properties of the  function $v$.

\begin{proposition}\label{prop-ave-1} The inequality
\begin{equation}\label{eq-lap-2} 
0 \le v \le \Delta v
\end{equation}
holds on $\Ha$. 
\end{proposition}
\begin{proof}
Since the hyperbolic laplacian $\Delta$ is invariant under the isometries of $\Ha$ it follows that 
$$
\int\limits_{B_{2}(x)} \, \Delta u^2(y)\, dA=\Delta v(x).
$$
Combining this with Proposition \ref{prop-lap-1} yields (\ref{eq-lap-2}).
\end{proof}

\begin{proposition}\label{prop-ave-2} There exists $C>0$ such that  
\begin{equation}\label{eq-lap-3}
|\nabla v(x)| \le C||u^2||_{C^{0}(B_{3}(x))}
\end{equation}
for every $x\in \Ha$ assuming  $\epsilon$ is small enough.
\end{proposition}

\begin{proof}

Let $x_1\in \Ha$ be such that $\dis_{\Ha}(x,x_1)<1$. We can choose an isometry $Q$ of $\Ha$ such that $Q(x)=x_1$, and so that 
\begin{equation}\label{eq-LL}
\max_{y\in B_{2}(x)} \dis_{\Ha}(y,Q(y))\le L\dis_{\Ha}(x,x_1)
\end{equation}
for some universal constant $L>0$. Then 
\begin{align*}
|v(x_1)-v(x)|&\le \int\limits_{B_{2}(x)} \,  |u^2(Q(y))-u^2(y)| \, dA(y)\\
&\le (1+o(1))\int\limits_{B_{2}(x)} \,  |\nabla u^2(y)|  \dis_{\Ha}(Q(y),y)\, dA(y),
\end{align*}
where $o(1)\to 0$ when $x_1\to x$. Combining this with (\ref{eq-LL}) gives
$$
\frac{|v(x_1)-v(x)|}{\dis(x,x_{1})} \le L \int\limits_{B_{2}(x)} \,  |\nabla u^2(y)| \, dA(y),
$$
which yields 
$$
|\nabla v(x)| \le L \max_{y\in B_{2}(x)} |\nabla u^2(y)|.
$$
Combining this with (\ref{eq-laba}) proves (\ref{eq-lap-2}). 
\end{proof}

\begin{proposition}\label{prop-ave-3} There exist $C,\eta>0$ such that when $\epsilon$ is small enough the estimate 
\begin{equation}\label{eq-lala}
\la^2(x)\le C v(x)
\end{equation}
holds for every $x\in \Ha$ assuming  $||u^2||_{C^0(B_{1}(x))}<\eta$.
\end{proposition}
\begin{proof}
Applying first (\ref{eq-laba'}), and then (\ref{eq-svidj}), we get
$$
\la^2(x)  \le  C ||u^2||_{C^0(B_{1}(x))}\le  C \int\limits_{B_{2}(x)} \, u^2(y)\, dA=Cv(x),
$$
which proves (\ref{eq-lala}).

\end{proof}

\section{Rigid subsurfaces of a minimal surface}
Let $\Omega \subset S$ be a subsurface with geodesic boundary. In this section we show that if the lift of $f(\Omega)$ to ${\Ho}$ is closely aligned with a geodesic plane $P\subset {\Ho}$, then the size of the integral of the principal curvature is controlled by the length of the boundary of $\Omega$. This is the content of Lemma \ref{lemma-close} which plays the key part in the proof of Lemma \ref{lemma-comp-1}. 

\begin{definition}\label{def-rigid} Let $f\from S\to \M$ be a minimal map, and $\Omega \subset S$ an essential subsurface with geodesic boundary.  We say that $\Omega$ is $f$-rigid if 
\begin{itemize}
\item  there exists $D>0$,
 
\item there exists a geodesic plane $P\subset {\Ho}$,
\end{itemize}
such that  $\dis_\M(\wt{f}(x),P)<D$ for  every point  $x\in \wt{\Omega}$. Here $\wt{\Omega}$ is a connected component of the lift of $\Omega$ to the universal cover $\Ha$.
\end{definition} 

\begin{lemma}\label{lemma-close} There exists $C>0$ such that for every $\epsilon$-nearly geodesic subsurface the estimate
\begin{equation}\label{eq-vreme}
\int\limits_{\Omega} \la^2 \,dA\le C  |\partial \Omega| \epsilon^2
\end{equation}
holds providing  that $\Omega$ is $f$-rigid, and $\epsilon$ is small enough.
\end{lemma}

\subsection{Rigid subsurfaces are $\epsilon$ rigid} In Definition \ref{def-rigid} we assume that the minimal subsurface $\wt{f}(\wt{\Omega})$ is $D$ away from a geodesic plane $P$. The purpose of this subsection is to promote $D$ to $\epsilon$, meaning that if we assume that  $\wt{f}(\wt{\Omega})$ is $D$ away from $P$, then it is actually $C\epsilon$ away from $P$.

Let $G<\pi_1(S)$ be a subgroup corresponding to the subsurface $\Omega\subset S$. Then $f_*(G)<\pi_1(\M)<\PSLC$. Let $\wt{\Omega}$ be the component of the lift of $\Omega$ which is invariant under $f_*(G)$. 
Since $\Omega$ is $f$-rigid, it follows that the plane $P$ is invariant under the group $f_*(G)$.
This implies that  the function $u\from \wt{\Omega}\to \R$ defined by (\ref{eq-u}) is equivariant and therefore well defined on $\Omega$. 
\begin{proposition}\label{prop-prom} There exists $C>0$ such that 
\begin{equation}\label{eq-pink-1}
||u^2||_{C^{0}(\Omega)}\le C\epsilon^2
\end{equation}
when $\epsilon$ is small enough.
\end{proposition}
\begin{proof}
Since $\Omega$ is a subsurface with geodesic boundary, it follows that $\Omega$ is equal to its convex core. Thus, 
every $x\in \wt{\Omega}$ belongs to a geodesic $\alpha$ which is entirely contained in $\wt{\Omega}$.

Let $\beta\subset {\Ho}$ be the  geodesic with the same endpoints as $\wt{f}(\alpha)$. Since $\Omega$ is $f$-rigid it follows that $\beta$ is a finite distance away from the plane $P$. But then $\beta\subset P$ since $P$ is a geodesic plane. On the other hand, from Proposition \ref{prop-rata} we have that $\dis_{\M}(f(\alpha),\beta)\le C_1\epsilon$, where $C_1$ is the constant from  Proposition \ref{prop-rata}.  Therefore, we derived the inequality
$$
\dis_{\M}(f(x),P)\le C_1\epsilon
$$
for every $x\in \wt{\Omega}$. But $\sinh(C_1\epsilon)\le 2C_1\epsilon$\, for $\epsilon$ small enough. Thus, (\ref{eq-pink-1}) holds for  $C=4C_1^2$.

\end{proof}

\subsection{Proof of Lemma \ref{lemma-close}}  We start by defining
$$
\Omega'=\{x\in \Omega: \dis_{\Omega}(x,\partial \Omega)\ge 3\}.
$$
Observe  that there exists a universal constant $L$ such that 
\begin{equation}\label{eq-lutam-1}
|\Omega\setminus \Omega'|\le L|\partial \Omega|,
\end{equation}
and
\begin{equation}\label{eq-lutam}
|\partial \Omega'|\le L|\partial \Omega|.
\end{equation}
Furthermore, without loss of generality we may assume that $\partial \Omega'$ is a union of finitely many smooth curves (the point is that $\Omega'$ can be approximated by such domains). 

We first estimate the integral on the left hand side of (\ref{eq-vreme}) over $\Omega\setminus \Omega'$.
Let $C_0$ be the constant from Theorem \ref{thm-seppi}. Then from (\ref{eq-lutam-1}) we get
\begin{equation}\label{eq-vreme-1}
\int\limits_{\Omega\setminus\Omega'} \la^2 \,dA\le C_0 |\Omega\setminus \Omega'| \epsilon^2\le 
C_0L |\partial \Omega| \epsilon^2.
\end{equation}
It remains to estimate this integral over $\Omega'$ which we do in the remainder of the proof.

Let $\numb$ be the inward-pointing  unit normal vector field on the  boundary $\partial \Omega'$. Green's formula gives
\begin{equation}\label{eq-green}
\int\limits_{\Omega'} \Delta v \,dA=\int\limits_{\partial \Omega'} \frac{\partial v}{\partial \numb } \,d\len,
\end{equation}
where  $d\len$ is the hyperbolic length form.

Firstly, let $C_1$, and $\eta_1$, be the constants from  Proposition \ref{prop-ave-3}. When $x\in \Omega'$, then 
$B_1(x)\subset \Omega$. Thus,   by (\ref{eq-pink-1}) the inequality
$$
||u^2||_{C^{0}(\Omega)}\le \eta_1
$$
holds when $\epsilon$ is small enough. Then from Proposition \ref{prop-ave-3} we derive the inequality
\begin{equation}\label{eq-pk-1}
\la^2(x)\le C_1 v(x),\,\,\,\,\,\,\,   x\in \Omega'.
\end{equation}
Combining this with the fact that $v\le \Delta v$, we obtain the inequality
\begin{equation}\label{eq-home}
\int\limits_{\Omega'} \la^2 \,dA \le  C_1 \int\limits_{\Omega'} \Delta v \,dA
\end{equation}
assuming $\epsilon$ is small enough.

On the other hand, let $C_2$ be the constant from Proposition \ref{prop-ave-2}. Then 
\begin{align}\label{eq-fait}
\int\limits_{\partial \Omega'} \frac{\partial v}{\partial \numb } \, d\len &\le \int\limits_{\partial \Omega'} |\nabla v| \, d\len 
\le \max_{x\in \partial \Omega'} |\nabla v(x)|  |\partial \Omega'|   \notag \\ 
\\
&\le C_2 ||u^2||_{C^{0}(\Omega)}  |\partial \Omega'|   \le C_2 C_3 |\partial \Omega'| \epsilon^2  \notag
\end{align}
where $C_3$ is the constant from Proposition \ref{prop-prom}. Replacing (\ref{eq-home}) and (\ref{eq-fait}) in the Green's formula (\ref{eq-green}) proves
$$
\int\limits_{\Omega'} \la^2 \,dA\le C_1C_2C_3  |\partial \Omega'| \epsilon^2\le  LC_1C_2C_3  |\partial \Omega| \epsilon^2
$$
where we applied (\ref{eq-lutam}) in the last inequality. Together with (\ref{eq-vreme-1}) this gives
$$
\int\limits_{\Omega} \la^2 \,dA\le L(C_0+C_1C_2C_3)  |\partial \Omega'| \epsilon^2.
$$
This proves the lemma by letting $C=L(C_0+C_1C_2C_3)$.

\section{Proof of Lemma \ref{lemma-comp-1} }\label{section-comp-1}


\subsection{Smoothing out the boundary of $S(\delta_1(h),h)$} In the remainder of this section we fix $h\in \N$.
To alleviate the notation we let $\delta_0=\delta_0(h)$,  $\delta_1=\delta_1(h)$, and $\epsilon_0=\epsilon_0(h)$ (see Remark \ref{remark-prva}).

\begin{proposition}\label{prop-druga} Let  $f\from S\to \M$ be any  $\epsilon$-nearly geodesic minimal surface.  There exists a (possibly disconnected) subsurface  $R \subset S$ (with piecewise smooth boundary) with the following properties:
\begin{enumerate}
\item  $S(\delta_1,h)\subset  R \subset S(\delta_0,h)$,
\item $|\partial R| \le 25 \big(|S|- |S(\delta_1,h)|\big)$,
\end{enumerate}
assuming $\epsilon<\epsilon_0$. 
\end{proposition}
\begin{proof}
We apply Lemma \ref{lem-smoothing} to $S(\delta_0,h)$ to obtain $R$ such that 
$$S(\delta_1,h) \subset R \subset \Ne_1(S(\delta_1,h)) \subset S(\delta_0,h),$$
and 
$$|\partial R| \le 25 |\Ne_1(\partial S(\delta_0,h))| \le 25 \big( |S| - |S(\delta_1,h)| \big). \qedhere$$
\end{proof}

\subsection{The endgame} Let $R_1,...,R_k$ be the connected components of $R$. 
By $\Omega_i\subset S$ we denote  the convex core of $R_i$. In other words, $\Omega_i$ is the subsurface whose boundary curves are geodesics homotopic to the homotopically non-trivial components of $\partial R_i$.

\begin{proposition}\label{prop-azra} Each $\Omega_i$ is $f$-rigid.
\end{proposition}
\begin{proof} Let $\wt{R}_i$ be a component of the lift of $R_i$ to $\Ha$. Then there exists a component $\wt{\Omega}_i$ of the lift of 
$\Omega$ such that every point on  $\wt{\Omega}_i$ is at most distance $D_1>0$ away from $\wt{R}_i$. On the other hand, since $R_i\subset S(\delta_0,h)$ it follows from Proposition \ref{prop-delta}  that for each point  $p\in \wt{R}_i$ there exists a unique 
connected component of $\wt{\HH}$ which is $\delta_0$ away from $\wt{f}(x)$. Since $\wt{R}_i$ is connected, this component is the same for all $x\in \wt{R}_i$, and we conclude that there exists a geodesic plane $P_i\subset {\Ho}$ such that every point on 
$\wt{R}_i$ is $\delta_0$ away from $P_i$. Thus, every point on $\wt{\Omega}_i$ is $(D_1+\delta_0)$ away from $P_i$ and we are done.
\end{proof}

We apply the Gauss-Bonnet formula and  get
\begin{align}\label{eq-o}
4\pi(\gen(S)-1)&=\int\limits_{S} (1+\la^2(x))\,dA_\sigma=\int\limits_{S} \,dA_\sigma+ \int\limits_{S} \la^2(x)\,dA_\sigma \le \notag  \\
\\
&\le \Area_\sigma(S)\left(1+\frac{2}{|S|}\int\limits_{S} \la^2(x)\,dA \right), \notag 
\end{align}
where in the last inequality we used the inequality (\ref{eq-ahlfors}) which implies  that $\frac{1}{2}dA\le dA_\sigma\le dA$ when $\epsilon$ is small enough. It remains to find a suitable upper bound for the last term in (\ref{eq-o}).

Let $C_1$ and $C_2$ be the constants from  Lemma \ref{lemma-close} and Theorem \ref{thm-seppi} respectively. Then
\begin{align*}
\int\limits_{S} \la^2(x)\,dA&= \sum_{i=1}^k \int\limits_{\Omega_i} \la^2(x)\,dA+ \int\limits_{S\setminus \Omega} \la^2(x)\,dA
\\
&\le C_1 \sum_{i=1}^k |\partial \Omega_i| \epsilon^2+ C^2_2|S\setminus \Omega| \epsilon^2
\end{align*}
where  $\Omega=\Omega_1\cup \cdots \cup \Omega_k$.
Since  $\sum_{i=1}^k |\partial \Omega_i| \le  |\partial R|$, from the previous inequality we get
\begin{align*} 
\int\limits_{S} \la^2(x) \,dA  &\le C_1 |\partial R| \epsilon^2+ C^2_2 \big( |S|-|\Omega | \big) \epsilon^2 \\
&\le C_1 |\partial R| \epsilon^2+ C^2_2  \left( (|S|-|R|)+\big||R|-|\Omega |\big|\right)\epsilon^2  \\
&\le  C_1 |\partial R| \epsilon^2+ C^2_2  \left( (|S|-|S(\delta_1,h)|)+|\partial R|\right)\epsilon^2, 
\end{align*}
where in the last inequality we used the fact that $S(\delta_1,h)\subset R$, and also Proposition \ref{prop-iso-1} to estimate $\big||R|-|\Omega |\big|$. 

Let $C_3=25$ be the constant from Proposition \ref{prop-druga}. We now use the estimate 
$|\partial R|\le C_3(|S|-|S(\delta_1,h)|)$, and get
\begin{align*} 
\int\limits_{S} \la^2(x) \,dA  &\le  C_1C_3 \big(|S|-|S(\delta_1,h)|\big) \epsilon^2+ C^2_2  
\left( \big(|S|-|S(\delta_1,h)|\big)+ C_3 \big(|S|-|S(\delta_1,h)|\big) \right)\epsilon^2\\
&= \big(C_1C_3+C^2_2+C^2_2C_3\big)\big(|S|-|S(\delta_1,h)|\big)\epsilon^2.
\end{align*}
Thus,
$$
\frac{1}{|S|}\int\limits_{S} \la^2(x) \,dA \le C_4\left(1-\frac{|S(\delta_1,h)|}{|S|}\right)\epsilon^2.
$$
Replacing this back into (\ref{eq-o}) we get 
$$
4\pi(\gen(S)-1)\le \Area_\sigma(S)\left(1+ 2   C_4\left(1-\frac{|S(\delta_1,h)|}{|S|}\right)\epsilon^2 \right),
$$
and the lemma is proved.

\section{The proof of Theorem \ref{thm-main-top}} \label{sec-main-top}
We start by defining the set of minimal surfaces which have a definite part of  their area located outside a $\delta_0 \equiv \delta_0(h)$ neighbourhood of the locus $\hh$ of totally geodesic surfaces of genus at most $h$. (We remind the reader that $\delta_0$ was defined in Section \ref{subsec-geod}.)

We remind the reader of our standing assumption  that $\epsilon\le \wh{\epsilon}$, where $\wh{\epsilon}>0$ is the universal constant from Proposition \ref{prop-us}. This ensures that each conjugacy class  $\Sigma\in \Su_\epsilon$ corresponds to a unique minimal surface.

To simplify our constructions we will work mostly with one genus at a time; accordingly we let $\Su_\epsilon(\gen) \subset \Su_\epsilon(g)$ be the surfaces with exactly genus $\gen$, so that
$$\Su_\epsilon(g) = \bigcup_{\gen \le g} \Su_\epsilon(\gen).$$
The reader should not be overly concerned with this abuse of notation, because after this section we will talk only about $\Su_\epsilon(\gen)$.

\begin{definition} Let $0\le q\le 1$. We say $\Sigma\in \Su_\epsilon(g,h,q)\subset \Su_\epsilon(g)$\,\, if 
$$
\frac{|S(\delta_0, h)|}{|S|}\le 1-q
$$
where $f\from S\to\M$ is the $\epsilon$-nearly geodesic minimal map representing the class $\Sigma$. 
\end{definition}
We also define $\Su_\epsilon(\gen,h, q)$ in analogy to $\Su_\epsilon(\gen)$, so that 
$$\Su_\epsilon(g,h,q) = \bigcup_{\gen \le g} \Su_\epsilon(\gen,h,q).$$

In the remainder of this section we prove Theorem \ref{thm-main-top} assuming the following lemma, which will be proven in Section \ref{sec-big-count} after the introduction and discussion in Section \ref{sec-geo-comb}.
\begin{lemma}\label{lemma-baza} 
For all $q >0$ and  $c>0$, there exists $h$ and  $\epsilon_0=\epsilon_0(q,c)$  such that 
\begin{equation}\label{eq-baza}
\lim_{\gen \to \infty} \frac{|\Su_\epsilon(\gen, h, q)|}{(c\gen)^{2\gen}}=0
\end{equation}
for every $\epsilon<\epsilon_0$.
\end{lemma}
The reader can easily verify that this implies the unboldfaced version:
\begin{corollary} \label{cor-baza}
For all $q >0$ and  $c>0$, there exists $h$ and  $\epsilon_0=\epsilon_0(q,c)$  such that 
\begin{equation}\label{eq-baza-unbold}
\lim_{g \to \infty} \frac{|\Su_\epsilon(g,h,q)|}{(cg)^{2g}}=0
\end{equation}
for every $\epsilon<\epsilon_0$.
\end{corollary}

Combining Corollary \ref{cor-baza} with the M\"uller-Puchta result we derive the following proposition.

\begin{proposition}\label{prop-baza-cor} 
For every $0< q \le 1$, there exists $h$ such that we have
$$
\lim_{g\to \infty} \frac{|\Su_\epsilon(g,h,q)|}{|\Su_\epsilon(g)|}=0,
$$
when $\epsilon$ is small enough.
\end{proposition}

\begin{proof}  
We are assuming (throughout the paper) that $\M$ contains a totally geodesic surface;
we denote by
$\Sigma_0\in \Su_\epsilon(g_0)$ the corresponding conjugacy class (here $g_0= \gen(\Sigma_0)$).
By the counting result by M\"uller-Puchta \cite{m-p}, there exists $\beta>0$ such that the number of different degree $n$ covers of $\Sigma_0$ is at least  $(\beta g_n)^{2g_{n}}$, where $g_n=(n(g_0-1)+1)$ is the genus of degree $n$ covering surface. 

Suppose $g$ is large, and let $n$ be an integer such that $g_n\le g \le g_{n+1}$. Then 
$$
(\beta g_n)^{2g_{n}} \le |\Su_\epsilon(g_n)|\le |\Su_\epsilon(g)|.
$$
Since $g_n\ge (g-g_0)$, we derive the estimate 
\begin{equation}\label{eq-mq-1}
(\beta (g-g_0))^{2(g-g_{0})} \le |\Su_\epsilon(g)|.
\end{equation}

Let $c<\beta$ in Corollary \ref{cor-baza}. Then  using (\ref{eq-mq-1}) and (\ref{eq-baza-unbold}), we get
$$
\limsup_{g\to \infty}  \frac{|\Su_\epsilon(g,q)|}{|\Su_\epsilon(g)|}\le \limsup_{g\to \infty}
\frac{(cg)^{2g}}{(\beta (g-g_0))^{2(g-g_{0})}}=0
$$
when $\epsilon<\epsilon_0(q,c)$. This proves the proposition.
\end{proof}

\newcommand{\muL}{\mu_{\mathcal{L}}}

Let $\mu^\Top _\epsilon$ be an $\epsilon$-topological limiting measure on $\G_2(\M)$. That is, there exists a sequence $g_n\to \infty$ such that 
$$
\mu^\Top_\epsilon=\lim\limits_{g_{n}\to \infty} \frac{1}{|\Su_\epsilon(g_n)|} \sum_{\Sigma\in \Su_\epsilon(g_n)} \mu(\Sigma)
$$
(recall that $\mu(\Sigma)$ is the pushforward of the hyperbolic area measure from $S$  to $\G_2(\M)$).
Let $\mu^\Top$ be a limit of the $\mu^\Top _\epsilon$ as $\epsilon \to 0$.  We then have
$$
\mu^\Top=\mu_{\HH}+\muL
$$
where $\mu_{\HH}$ is supported in $\HH$, and $\mu_{\mathcal{L}}$ is a multiple of the Liouville measure.
We will assume that $\muL$ is nonzero and obtain a contradiction. 

We choose  $q=|\muL|/4$ and choose $h > 0$ to satisfy Propostion \ref{prop-baza-cor} with this value of $q$.

We set
$$
\wh{\Su}_\epsilon(g_n,h,q) =\Su_\epsilon(g_n)\setminus \Su_\epsilon(g_n,h,q).
$$
Then from Proposition \ref{prop-baza-cor} we conclude that 
\begin{equation}\label{eq-tot}
\mu^\Top_\epsilon=\lim\limits_{g_{n}\to \infty} \frac{1}{|\wh{\Su}_\epsilon(g_n,h,q)|} \sum_{\Sigma \in\wh{\Su}_\epsilon(g_n,h,q)} \mu(\Sigma).
\end{equation}
Let $\Sigma\in \wh{\Su}_\epsilon(g_n,h,q)$, and denote by $f\from S\to \M$ the minimal surface representing $\Sigma$. Then
\begin{equation}\label{eq-tot-1}
1-q\le \frac{|S(\delta_0, h)|}{|S|}.
\end{equation}
The inequality (\ref{eq-tot-1}) yields the estimate
$$
1-q\le \mu(\Sigma)\big(\Ne_{\delta_0}(\HH) \big) 
$$
for every $\Sigma\in  \wh{\Su}_\epsilon(g_n,h,q)$. Combining this with (\ref{eq-tot}) gives  the estimate
$$
1-q\le \mu^\Top_\epsilon(\cl{\Ne_{\delta_0}(\HH)}  )
$$
for every $\epsilon$-topological limiting measure on $\G_2(\M)$. In turn, this shows that
\begin{equation}\label{eq-tot-2}
1-q \le \mu^\Top(\cl{\Ne_{\delta_0}(\HH)})
\end{equation}
for every topological limiting measure $\mu^\Top$ on $\G_2(\M)$. 

On the other hand, 
by the first property of $\delta$ in Propostion \ref{prop-delta},
taking our particular limiting measure $\mu^\Top = \mu_\HH + \muL$,
we have
\begin{align*}
\mu^\Top\big(\HH \setminus \cl{\Ne_{\delta_0}(\HH)} \big) 
&\ge \muL\big(\G_2(\M) \setminus \cl{\Ne_{\delta_0}(\HH)}\big)  \\
&\ge \frac12 \muL(\G_2(\M)) = 2q > q.
\end{align*}
This and \eqref{eq-tot-2} are a contradiction.

%

\section{From geometry to combinatorial data}  \label{sec-geo-comb}
In this section we describe a set of allowable combinatorial data records, such that for any $\epsilon$-nearly geodesic $f\from S\to \M$,
we can produce such a data record, and also reconstruct $f$ from the data. This reduces the upper bound on $\Su(\epsilon, \gen, h)$ to an upper bound on the number of allowable data records, and this bound is provided in Section \ref{sec-big-count}.
\subsection{The original construction and count}
We begin with a summary of the proof of the upper bound in  \cite{k-m-1}.
While there is nothing mathematically new in this subsection,
it contains lemmas and definitions that will be used in the remainder of the paper. 

In the setting of \cite{k-m-1},
we also have a closed hyperbolic 3-manifold $\M$,
and we let $\Su(\gen)$ denote the set of conjugacy classes of genus $\gen$ surface subgroups of $\pi_1(\M)$.
We say that a $\pi_1$-injective map $f\from S \to \gen$ represents $\Sigma \in \Su(\gen)$
if $\Sigma = [f_*\pi_1(S)]$ (where $[f]$ is the homotopy class of $f$),
and observe that for each $\Sigma \in \Su(\gen)$, 
there is a closed hyperbolic surface $S$ of genus $\gen$ and a pleated surface (and hence 1-Lipshitz) map $f\from S \to \M$
that represents $\Sigma$.
 
The upper bound in \cite{k-m-1} then goes as follows.
\begin{theorem} \label{thm-basic-count}
There exists $C = C(\M)$ such that $|\Su(\gen)| \le (C\gen)^{2 \gen}$.
\end{theorem}

To prove Theorem \ref{thm-basic-count},
we begin by proving the following theorem, which appears as (or follows immediately from)
Lemma 2.1 of \cite{k-m-1}. Let $r_\M$ denote the injectivity radius of $\M$.
\begin{theorem} \label{thm-bounded-geom}
There exists an $L\equiv L(r_\M)$ such that any closed hyperbolic surface $S$ of injectivity radius at least $r_\M /2$
has a geodesic triangulation such that
\begin{itemize}
\item
the edges have length at most $\min(1, r_\M/20)$,
\item
there are at most $L\gen(S)$ vertices
\item
each vertex has degree at most $L$, and 
\item
there are no more than $L$ vertices in any ball of radius 1 in $S$. 
\end{itemize}
\end{theorem}
We will fix such an $L \equiv L(r_\M) \equiv L(\M)$ and call such a triangulation a bounded geometry triangulation. 

In Section 2.2 of \cite{k-m-1} we also prove the following:
\begin{lemma} \label{lem-old-e-tau}
Let $S$ be a closed hyperbolic surface of genus $\gen$ and let $\sigma$ be a graph embedded in $S$.
Then we can find $(\hat \tau, e(\tau))$ where 
\begin{enumerate}
\item
$\hat\tau$ is a spanning tree for $\sigma$, 
\item
$e(\tau)$ is an additional set of $2\gen$ ``distinguished'' edges of $\sigma$, and
\item
the inclusion of $\hat\tau \cup e(\tau)$ into $S$ is $\pi_1$-surjective.
\end{enumerate}
\end{lemma}
We call $\utau\equiv(\hat\tau, e(\tau))$ an \emph{effective graph pair} for $S$, 
and let $\tau = \hat\tau \cup e(\tau)$. 
We observe that Property 3 is implied by the inclusion being surjective on $H_1$, and implies that $S\setminus \tau$ is a union of disjoint polygons.

We can now abstract out the properties of an effective graph pair of a surface of genus $\gen$.
We say that $\utau=(\wh{\tau},e(\tau))$ is an (abstract) graph pair if $\wh{\tau}$ is an (abstract) tree, 
and $e(\tau)$ an additional set of edges whose endpoints are vertices of $\wh{\tau}$. 
We call $\wh\tau$ the \emph{spanning tree} of $\utau$, and $e(\tau)$ the \emph{distinguished edges} of $\utau$, 
and let $\tau = \wh\tau \cup e(\tau)$.

\begin{definition} We say that a  graph pair $\utau=(\wh{\tau},e(\tau))$ is a $\gen$-polygonalization if:
\begin{enumerate}
\item the set $e(\tau)$ has exactly $2\gen$ elements,
 
\item the graph $\tau$ is equipped with a cyclic ordering of  edges round each vertex,
 
\item the ribbon graph $R(\tau)$ is a surface of genus $\gen$.
\end{enumerate}
We furthermore say that a $\gen$-polygonalization $\utau$ is $L$-bounded if  $\tau$ has at most $L\gen$ vertices, and each vertex has degree at most $L$. 
\end{definition}

\begin{remark} \label{rem-ribbon}
Recall that  the ribbon graph $R(\tau)$  is obtained by thickening the edges of the graph $\tau$. In particular, $R(\tau)$ is a surface with (polygonal) boundary; we can then add polygons to this boundary to form a closed surface.  
So we can reconstruct our surface (with an effective graph pair) from the abstract $\gen$-polygonalization. 
Moreover,
if a $\gen$-polygonalization is embedded into a closed surface of genus $\gen$ (so as to preserve the cyclic ordering around the vertices), then the complementary components are polygons, canonically associated to the boundary components of $R(\tau)$. 
\end{remark}

In \cite{k-m-1} we then choose a cover $\Dol$ of $\M$ (with small open balls), and use it to turn the effective graph pair into combinatorial data. Accordingly, given $\Dol$,
we have the following definition. 
\begin{definition}  We say that the pair  $(\utau,\lab)$ is a $\Dol$-valued $\gen$-polygonalization if 
$\utau$ is a $\gen$-polygonalization, and $\lab\from V(\tau)\to \Dol$ a map (colouring). Here $V(\tau)$ is the set of vertices of $\tau$.
\end{definition}
We denote the set of all $\Dol$-valued $L(\M)$-bounded $\gen$-polygonalizations by $\Tr^\gen_\Dol$ (where $L(\M)$ is given by Theorem \ref{thm-bounded-geom}).


%
%
%
%
%
%
%
%
%

\begin{definition}  Suppose that $\Sigma \in \Su(\gen)$. 
We say that $\Sigma$ is compatible with  $(\tau,\lab)\in \Tr^\gen_\Dol$
if there exists a 2-Lipschitz map $f\from S \to \M$ representing $\Sigma$ and an associated embedding $b\from \tau\to S$
such that $b(\tau)$ has bounded geometry
and $f(b(v)) \in \lab(v)$ for all $v\in V(\tau)\subset S$.
\end{definition}

Given any $\Sigma \in \Su(\gen)$,
we can find a pleated surface map $f\from S\to \M$ representing $\Sigma$,
find a bounded geometry triangulation of $S$ by Theorem \ref{thm-bounded-geom},
and then find an effective graph pair $\utau$ for $S$ by applying Lemma \ref{lem-old-e-tau} to this triangulation.
We can then define a labelling $\lab\from V(\tau)\to \Dol$ by letting $\lab(v)$ be any $D \in \Dol$ that contains $f(v)$;
the result is an element of $\Tr^\gen_\Dol$ that is compatible with $\Sigma$.

We then observe the following, which was essentially shown in Section 2 of  \cite{k-m-1}.
\begin{proposition}\label{prop-tri-1} Suppose 
$$
\max_{D\in \Dol}\,\, \operatorname{diam} D \le \frac{\inj_{\M}}{10}. 
$$
Then 
each $(\tau,\lab)\in \Tr^\gen_\Dol$ is compatible with at most one $\Sigma \in \Su_\epsilon(\gen)$.
\end{proposition}
\begin{remark} Although each $(\tau,\lab)\in \Tr^\gen_\Dol$ is compatible with at most one  $\Sigma$, a given $\Sigma$ can be compatible with many pairs $(\tau,\lab)\in \Tr^\gen_\Dol$.
\end{remark} 

\begin{proof}[Proof of Proposition \ref{prop-tri-1}]
Suppose, for $i=1,2$, that $f_i\from S_i \to \M$ is an $\epsilon$-nearly geodesic minimal  map,  and   $(\tau,\lab)$  a $\Dol$-polygonalization which is compatible with each $f_i$, and let $b_i\from  \tau \to S_i$, $i = 1, 2$ be an associated embedding (for the compatibility with $f_i$).  
As observed in Remark \ref{rem-ribbon},
the complementary components to the images of the $b_i$ are polygons, 
so we can find a homeomorphism $h\from S_1\to S_2$ such that $h\circ b_1 = b_2$.
Moreover, 
for every $v\in V(\tau)$,
the points $f_1(b_1(v))$ and $f_2(h(b_1(v))$ belong to a common $D \in \Dol$. 
Then by  Lemma 2.4 in 
\cite{masters}, the maps $f_1$ and $f_2   \circ h$ are homotopic%
.
Thus, they determine the same conjugacy class $\Sigma\in \Su_\epsilon(\gen)$.
\end{proof}
Since every $\Sigma \in \Su(\gen)$ is compatible with at least one $(\utau, \lab)$,
and each $(\utau, \lab)$ is compatible with at most one $\Sigma$, it follows that 
\begin{equation}
|\Su(\gen)| \le |\Tr^\gen_\Dol|.
\end{equation}
A simple combinatorial argument then shows that $|\Tr^\gen_\Dol| \le (C_\M \gen)^{2 \gen}$, which completes the proof of Theorem \ref{thm-basic-count}.

\subsection{The new ingredients}
We now introduce the refinements needed to prove Lemma \ref{lemma-baza}.
First, 
we must choose a cover $\Dol$,
not of $\M$,
but of $\G_2(\M)$.
Moreover,
it must be a cover with wafer-thin disks, which we now make precise.

Let $(y,\Pi)\in \G_2(\M)$.  We define the set $\Hull_{r,\epsilon}(y,\Pi)\subset \G_2(\M)$,  called the $(r,\epsilon)$-hull of $(y,\Pi)$, by letting
$$
\Hull_{r,\epsilon}(y,\Pi)=\bigcup_{f} f\big(B_r(x)\big),
$$
where $(f, x)$ varies over all $\epsilon$-nearly geodesic maps $f\from S\to \M$ with $x\in S$ and $f(x)=(y,\Pi)$. 
We define the $(r,\epsilon)$-hull of a subset $D\subset \G_2(\M)$ by
$$
\Hull_{r,\epsilon}(D)=\bigcup_{(y,\Pi)\in D} \Hull_{r,\epsilon}(y,\Pi).
$$

\begin{definition}\label{definition-akses} We say that  sets  $D_1,D_2 \subset \G_2(\M)$ are mutually  $\epsilon$-accessible if
$\Hull_{1,\epsilon}(D_1)\cap D_2\neq \emptyset$.
\end{definition}

The existence of our desired cover $\Dol$ is summarized in the following, which we will prove in Section \ref{section-Dol}. 
\begin{lemma} \label{lemma-Dol}
There exists $C\equiv C_{\ref{lemma-Dol}}(r_\M) \in \N$ such that
for all $c > 0$ there exists $h \in \N$ such that for all $\delta > 0$, 
there exist $\epsilon > 0$
and a finite open covering $\Dol$ of $\G_2(\M)$,  with the following properties for all $D \in \Dol$:
\begin{enumerate}
\item 
$\text{diam}_\M\big(\xi(D)\big)\le \frac{\inj_{\M}}{10}$, 
\item \label{it-mapped-C}
there are at most $C$ elements of $\Dol$ that are mutually $\epsilon$-accessible with $D$,  
\item \label{it-c-small}
if $D\not\subset \Ne_{\delta}(\hh)$
and $f\from S\to \M$ is an $\epsilon$-nearly geodesic minimal map,
then  $\mu_f(\cl{\Hull_{1,\epsilon}(D)}) \le c$.
\end{enumerate}
\end{lemma}
A finite covering $\Dol$ satisfying the conditions of the lemma is called a $(c,\delta)$-covering. 
The crucial new property is Property \ref{it-c-small};
when $\hh = \HH$,
this follows from $D$ being a wafer-thin neighborhood of some $P_r(p)$ that is disjoint from $\HH$. Here $P_r(p)\subset \G_2(\M)$ denotes the geodesic disc 
of radius $r$ that is tangent to the pointed plane $p$ at $0$.

We can naturally relate the properties of a cover $\Dol$ (such as the one given in Lemma \ref{lemma-Dol}) to properties of a compatible $\Dol$-labeled triangulation.
\begin{proposition}\label{prop-c} 
Suppose $f\from S\to \M$ is an $\epsilon$-nearly geodesic minimal map,
$\tau$ is a bounded geometry graph on $S$ (e.g.\ a subgraph of a bounded geometry triangulation),
$\Dol$ is a cover of $\G_2(\M)$, and $\lab\from V(\tau)\to \Dol$ satisfies $f(v) \in \lab(v)$ for all $v \in V(\tau)$. 
Then 
\begin{enumerate}
\item \label{it-supp-mut}
if $v_1$ and $v_2$ are two vertices connected by an edge from $\tau$,
then  $\lab(v_1), \lab(v_2) \in \Dol$ are mutually $\epsilon$-accessible.
\item \label{it-mapped}
for any $D \in \Dol$, we have
$$
|V_D|  \le 4\pi L \gen \mu_f(\Hull_{1, \epsilon}(D)),
$$
where $V_D$ is the set of vertices of $\tau$ that are mapped by $\lab$ to $D$. 
\end{enumerate}
\end{proposition}
\begin{proof} 
For Statement \ref{it-supp-mut},
let $D_i = \lab(v_i)$,  so $f(v_i) \in D_i$, for $i = 1, 2$.
Since $\dis_S(v_1,v_2)\le 1$ it follows that $\Hull_{1,\epsilon}(D_1)\cap D_2\neq \emptyset$. Therefore $D_1$ and $D_2$ are mutually $\epsilon$-accessible.

Let us now prove Statement \ref{it-mapped} for a given $D \in \Dol$. 
We have
\begin{equation}\label{eq-spir}
\bigcup_{v\in V_D}  f\big(B_{1}(v)\big) \, \, \subset\,\, \Hull_{1, \epsilon}(D).
\end{equation}
Set
$$
X_D=\bigcup_{v\in V_{D}(\tau)}  B_{1}(v).  
$$
Then \eqref{eq-spir} implies
\begin{equation}\label{eq-spir-1}
|X_D|\le \mu_f(\Hull_{1, \epsilon}(D))|S|.
\end{equation}

On the other hand, we know the number of vertices of $\tau$ in any ball of radius $1$ is at most $L$.
Therefore every point $x\in X_D$ belongs to at most $L$  different balls $B_1(v)$, $v\in V_D$. We conclude 
$$
|V_D|\le L|X_D|\le Lc|S| \le 4\pi cL\gen,
$$
where $c := \mu_f(\Hull_{1, \epsilon}(D))$.
\end{proof}

The second new ingredient needed to prove Lemma \ref{lemma-baza} is the following refinement of Lemma \ref{lem-old-e-tau}, 
which we will prove in Section \ref{sec-proof-tri}.
\begin{lemma}\label{lemma-tri} 
For every $h \in \N$ and $0<q\le 1$ there exists $\delta\equiv\delta_{\ref{lemma-tri}}(h, q)>0$ and $\epsilon_0 > 0$  with the following properties.
Suppose $f\from S \to \M$ is the minimal map representing $\Sigma \in \Su_\epsilon(g, h, q)$.
Then we can find an effective graph pair $(\hat\tau, e(\tau))$ for $S$ 
such that $q\gen$ of the edges in $e(\tau)$ lie outside of $S(\delta, h)$. 
\end{lemma}

\subsection{Proportionately compatible polygonalizations}
We now refine  our notion of a compatible $\Dol$-valued $\gen$-polygonalization as follows. 
Given any cover $\Dol$ of $\G_2(\M)$,
we let  $\Dol'_{h, q}$ be the $D\in\Dol$ which are not subsets of $\Ne_\delta(\HH_h)$, where $\delta = \delta_{\ref{lemma-tri}}(h, q)$. 

\begin{definition}
Suppose that $\Sigma \in \Su_\epsilon(\gen,h,q)$.
We say that $(\utau,\lab) \in \Tr^\gen_\Dol$ is proportionately compatible with $\Sigma$ if
it is compatible with $\Sigma$ and there are at least $q\gen$ edges in $e(\tau)$ whose endpoints are both mapped by $\lab$ to elements of $\Dol'_{h, q}$. 
\end{definition}

This is a mouthful,
but we have an immediate corollary to Lemma \ref{lemma-tri}:
\begin{proposition} \label{prop-q-compat}
For every $h \in \N$ and $q > 0$, 
there exists $\epsilon_0$ such that for all $\epsilon < \epsilon_0$,
any $\Dol$,
and $\Sigma \in \Su_\epsilon(\gen, h, q)$,
there exists $(\utau, \lab) \in \Tr^\gen_\Dol$ that is proportionately compatible with $\Sigma$.
\end{proposition}
%

We say that  $(\utau,\lab)\in \Tr^\gen_\Dol$ is a $(q, h, \gen, \epsilon)$-allowable $\Dol$-valued polygonalization 
if it is proportionately compatible with some $\Sigma \in \Su_\epsilon(\gen,h,q)$,
and we denote by $\Tr^\gen_\Dol(\epsilon,h,q)\subset  \Tr^\gen_\Dol$
the set of all $(q, h, \gen, \epsilon)$-allowable $\Dol$-valued polygonalizations.
Since every $\epsilon$-nearly geodesic minimal map is 2-Lipshitz for $\epsilon$ small enough,
we can combine Proposition \ref{prop-tri-1} and Proposition \ref{prop-q-compat} to immediately conclude:
\begin{proposition}\label{prop-estim} For any $\Dol$,  the estimate
\begin{equation}\label{eq-jasta}
|\Su_\epsilon(\gen,h,q)|\le |\Tr^\gen_\Dol(\epsilon,h,q)|
\end{equation}
holds when $\epsilon < \epsilon_0(h, q)$.
\end{proposition}
We will show Lemma \ref{lemma-baza} in Section \ref{sec-big-count} by estimating the size of $\Tr^\gen_\Dol(\epsilon,h,q)$ for a suitable choice of $\Dol$.

\section{Counting allowable $\Dol$-valued polygonalizations} \label{sec-big-count}
Now suppose we are given $c_0 > 0$ and $q \in (0, 1]$.
We let $h = h_{\ref{lemma-Dol}}(c_0)$
(so $h$ is the $h$ provided in \ref{lemma-Dol} that depends on the $c$ in \ref{lemma-Dol}, that we take to be $c_0$),
and let $\delta = \delta_{\ref{lemma-tri}}(h,q)$,
and then let $\epsilon_0$ and $\Dol$ be given (in terms of $c_0$ and $\delta$) by Lemma \ref{lemma-Dol}, and take any $\epsilon < \epsilon_0$.
We will show that for any $c > 0$,
and an appropriate choice of $c_0$,
that $|\Tr^\gen_\Dol(\epsilon,q,h)| = o((c\gen)^{2\gen})$ as $\gen \to \infty$, 
which implies Lemma \ref{lemma-baza}.

We let $\Dol' \subset \Dol$ be those $D \in \Dol$ for which $D \not\subset \Ne_\delta(\hh)$.
Before computing our upper bound for $|\Tr^\gen_\Dol(\epsilon,q,h)|$, we will summarize the properties of every $(\tau, \lab) \in \Tr^\gen_\Dol(\epsilon,q, h)$.
\begin{enumerate}
\item   \label{it-upper-bound}
The graph $\tau$ has at most $L\gen$ vertices.
\item  \label{it-degree}
The degree of every vertex of $\tau$ is at most $L$.
\item \label{it-mut}
If $v_1$ and $v_2$ are two vertices connected by an edge from $\tau$ then  $\lab(v_1), \lab(v_2) \in \Dol$ are mutually accessible.
\item \label{it-upper-mapped}
For every $D \in \Dol'$,
the number of vertices of $\tau$ which are mapped by $\lab$ to $D$ is at most $c_1 \gen$, where $c_1 = 4\pi L c_0$.
\item \label{it-q}
There are at least $q\gen$ distinguished edges from the set $e(\tau)$ whose endpoints are mapped by $\lab$ into $\Dol'$.
\end{enumerate}
Properties \ref{it-upper-bound} and \ref{it-degree} follow from the definition of an allowable $\gen$-polygonalization. Property \ref{it-mut} follows from Statement \ref{it-supp-mut} in Proposition \ref{prop-c}, 
and Property \ref{it-upper-mapped} follows from Property \ref{it-c-small} of $\Dol$, 
the definition of $\Dol'$, and Statement \ref{it-mapped} of Proposition \ref{prop-c}.
Property \ref{it-q} 
follows immediately from the definitions of $\Tr^\gen_\Dol(\epsilon,q, h)$ and $\Dol'$.

We can obtain every $\Dol$-valued $\gen$-polygonalization of bounded geometry  $(\tau,\lab) \in \Tr^\gen_\Dol(\epsilon,q,h)$ as follows:
\begin{itemize}
\item 
First choose a spanning tree $\wh{\tau}$. 
\item 
Then chose a colouring $\lab\from V(\wh{\tau})\to \Dol$. 
\item 
Then add the $2\gen$ distinguished edges to the  $\wh{\tau}$.
 \item 
Then equip the graph $\tau$ with the cyclic ordering around each vertex.
\end{itemize}
Of course our choices will be constrained by Properties 1--5 listed above. 
Let us write down the estimate that comes out from the  description above:
\begin{equation}\label{eq-abc}
|\Tr^\gen_\Dol(\epsilon,q,h)|\le A_1A_2A_3A_4, 
\end{equation}
where 
\begin{itemize}
\item 
$A_1$ is the number of different trees $T$ which agree with  the tree $\hat\tau$ of $\utau = (\hat\tau, e)$ of  some pair  $(\utau,\lab) \in \Tr^\gen_\Dol(\epsilon,q,h)$, 
\item 
$A_2$ is the (maximal) number of ways we can define the colouring $\lab\from V(\wh{\tau})\to \Dol$ for a fixed $\wh{\tau}$ (here $\wh{\tau}$ is a spanning tree of some pair  $(\utau,\lab) \in \Tr^\gen_\Dol(\epsilon,q,h)$),
\item  $A_3$ is the (maximal) number of different pairs $(\utau,\lab) \in \Tr^\gen_\Dol(\epsilon,q,h)$ which share the same spanning tree $\wh{\tau}$ and  colouring $\lab$,
\item 
$A_4$ is the number of  cyclic orderings of edges of $\tau$.
\end{itemize}

\subsection{Estimating $A_1$, $A_2$, $A_3$, and $A_4$}
We will compute upper bounds for $A_1$, $A_2$, $A_3$, and $A_4$ in turn.

The number of different unlabelled trees on $n$ vertices is $\le C_1 12^n$, for some universal constant $C_1 > 0$ (for example see \cite{k-m-1}). It follows from Property \ref{it-upper-bound} that 
$$
A_1 \le C_1 12^{L\gen}.
$$

Next, we bound  $A_2$. Thus, we consider a fixed  allowable tree $\wh{\tau}$, and estimate  the number of allowable maps 
$\lab:V(\wh{\tau})\to \Dol$. Choose a base vertex $v_0 \in V(\wh{\tau})$, and record $\wh{\tau}$ as a rooted tree. Then every vertex $v\in V(\wh{\tau})$ (besides the base vertex) has a predecessor denoted by $v'$.   

There are at most $|\Dol|$ choices for the value of $\lab(v_0)$.
On the other hand, if the value $\lab(v')$ has already been specified, then there are at most $C_2$ possible values for $\lab(v)$, where $C_2= C_{\ref{lemma-Dol}}$. This is because the open sets $\lab(v), \lab(v') \in \Dol$ are mutually accessible by Property \ref{it-mut}. We conclude that there are at most $C_2^{|V(\wh{\tau})|}\le C_2^{L\gen}$ ways of defining $\lab$ on the set $V(\wh{\tau})\setminus \{v_0\}$. Putting this all together yields the bound
$$
A_2 \le |\Dol| C_2^{L\gen}.
$$

We will now bound $A_3$.
We fix an allowable  tree $\wh{\tau}$, and an allowable map $\lab\from V(\wh{\tau})\to \Dol$.
By Properties 5 and 3, 
we can generate all $(\gen, \epsilon, h, q)$-allowable 
sets of distinguished edges $e(\tau)$
by generating a set $e_1$ of $\ceil{q \gen}$ edges that connect two mutually accessible $D$'s in $\Dol'$,
and then generating a set $e_2$ of $2\gen - \ceil{q \gen}$ edges that connect mutually accessible $D$'s in $\Dol$.

When we choose an edge in $e_1$, 
there are $|V(\tau)|$ ways to choose the first vertex $v$.
By Property \ref{it-mut} above and Property 2 of $\Dol$ (from Lemma \ref{lemma-Dol}),
there are at most $C_2$ ways to choose the label $\lab(w) \in \Dol'$ of the second vertex $w$,
and by Property \ref{it-upper-mapped} above 
there are at most $c_1\gen$ choices for $w$ given its label. 
Therefore
there  
are at most $(|V(\tau)| c_1C_2 \gen)^{\ceil{q \gen}}$ choices of sets $e_1$ along with an ordering of the edges,
and hence at most $(|V(\tau)| c_1C_2 \gen)^{\ceil{q \gen}}/ \ceil{q \gen}!$ ways of choosing the set $e_1$. 
Similarly,
but without any constraints on the choice of the second vertex,
there are 
$(|V(\tau)|^2)^{2\gen - \ceil{q \gen}}/(2\gen - \ceil{q \gen})!$ ways to choose $e_2$.
We therefore have
$$
A_3 
\le  \frac{ \big(|V(\tau)| c_1C_2 \gen\big)^{\ceil{q \gen}}
			\,  (|V(\tau)|^2)^{2\gen - \ceil{q \gen}}}   
	{\ceil{q \gen}! (2\gen - \ceil{q \gen})!}
\le 
 \frac{ \big(c_1C_2 L \gen^2 \big)^{\ceil{q \gen}}\, (L^2\gen^2)^{2\gen - \ceil{q \gen}}\, 2^{2\gen}\, 4^{2\gen} }
 	{\gen^{2\gen}}
$$
where we used the following estimates
$$
|V(\tau)| \le L\gen\qquad \left(\frac{\gen}{4}\right)^{2\gen} < (2\gen)! \qquad \frac{(a+b)!}{a!b!} \le 2^{a+b}.
$$
Compiling like terms we get
$$
A_3 \le C_4^{\gen} c_1^{q\gen} \gen^{2\gen}
$$
for some constant $C_4>0$.

It remains to bound $A_4$. We easily find (as in \cite{k-m-1})
$$
A_4\le \big(\text{deg}(\tau)!\big)^{|V(\tau)|}\le (L!)^{L\gen}.
$$

Now, replacing these estimates in (\ref{eq-abc}) we get
$$
|\Tr^\gen_\Dol(\epsilon,q,h)|\le C_5^\gen|\Dol|c_1^{q\gen} \gen^{2\gen}
$$
for some constant $C_5$, and every small enough $\epsilon$. Then for every $c>0$, and every fixed $c_1>0$, we have
$$
\frac{|\Tr^\gen_\Dol(\epsilon,q, h)|}{(c\gen)^{2\gen}}\le
\frac{C_5^\gen|\Dol|c_1^{q\gen} \gen^{2\gen}}{(c\gen)^{2\gen}}=\left(C_5\big(1+o(1)\big)\frac{c_1^q}{c^2} \right)^{\gen}
$$
where $o(1)$ is a function of $\gen$ which tends to zero when $\gen\to \infty$. 
Taking $c_1 < (c^2/C_5)^{1/q}$, we obtain the desired result.

\section{Proof of Lemma \ref{lemma-tri}} \label{sec-proof-tri}

In \cite{k-m-1} the authors constructed a bounded geometry triangulation, and then found an embedded graph pair and resulting $\Dol$-valued polygonalization. Here we assume that $|S(\delta_0(h), h)| \le (1-q)|S|$, and use it to construct an embedded graph pair with at least $qg$ edges lying outside of $S(\delta)$, for some $\delta$ that will depend only on $\hh$ and $q$. We will first provide an upper bound for the portion of $H_1(S)$ that is generated by $S(\delta)$, and then use that bound when we build the graph pair. In what follows, we will assume we are given a value of $h$ and write $\delta_0$ for $\delta_0(h)$ and $S(\delta)$ for $S(\delta, h)$. 

\subsection{The size of $\partial{S(\delta)}$ }
As usual,   $f\from S\to \M$ denotes an $\epsilon$-nearly geodesic minimal map.  We first prove a version of Proposition \ref{prop-druga} where  the resulting subsurface has small boundary compared to $|S|$. 

\begin{proposition}\label{prop-druga-new} 
For any $\alpha > 0$ there exist $\delta>0$, $\epsilon_0>0$, 
such that for every $\epsilon$-nearly geodesic minimal surface  $f\from S\to \M$ (with $\epsilon < \epsilon_0$), 
there exists a subsurface  
$R \subset S$ which satisfies
\begin{enumerate}
\item  
$\Ne_{10}(S(\delta))\subset  R \subset S(\delta_0)$,
\item 
$|\partial R| \le \alpha|S|$.
 \end{enumerate} 
\end{proposition}
\begin{proof}
We let $n = \ceil{25/\alpha}$.

Inductively applying Proposition \ref{prop-prva}, 
we see that there exists $\epsilon_0 = \epsilon_0(n) = \epsilon(\alpha)$
and 
$\delta_{2n}<\delta_{2n-1}<\cdots<\delta_1<\delta_0$, 
such that for every $\epsilon$-nearly geodesic minimal surface  $f\from S\to \M$ with $\epsilon < \epsilon_0$,
we have $\Ne_{10}(S(\delta_{k+1})) \subset S(\delta_k)$ (for $0 \le k < 2n$),
and hence, for $0 < k < 2n$, 
\begin{equation}\label{eq-baba}
\Ne_{10}(\partial (S(\delta_k))) \subset S(\delta_{k-1}) \setminus S(\delta_{k+1}).
\end{equation}
We show that the proposition holds for $\delta:=\delta_{2n}$ and the $\epsilon_0$ we have chosen.

Notice that the sets $S(\delta_{k-1})\setminus  S(\delta_{k+1})$ are mutually disjoint for different odd values of $k$. 
Since there are $n$ such sets, at least one of them has the area at most $|S|/n$. 
Thus we can choose $l$ odd such that  $|S(\delta_{l-1})\setminus  S(\delta_{l+1})|\le |S|/n$. 
By Lemma \ref{lem-smoothing} (applied to $S(\delta_l)$) and our choices of the $\delta_k$,  $l$, and $\alpha$,
we can find $R$ such that 
$$\Ne_{10}(S(\delta_{2n}))\subset S(\delta_l) \subset R \subset S(\delta_{l-1}) \subset S(\delta_0)$$ and 
$$|\partial R| \le 25 |S| / n \le \alpha |S|.$$

\end{proof}

\subsection{The homology of $S(\delta)$}
Let $S$ denote a closed hyperbolic surface. By $\sys(S)$ we denote the length of shortest closed geodesic on $S$ (note that $\sys(S)=2\inj(S)$).
If $T \subset S$ is closed, 
we let $\inbetti T S$ denote the dimension of the image of $H_1(T; \R)$ in $H_1(S; \R)$.
If $R \subset S$ is a subsurface we let $\numcomp(\partial R)$ denote the number of components of the boundary $\partial R$ of $R$.
We begin with the following corollary of the Gauss-Bonnet  theorem and the isoperimetric inequality.

\begin{proposition} \label{prop-bound}
There exists a constant $C>0$ with the following properties. Let   $S$ be a closed hyperbolic Riemann surface,  and
 suppose $R \subsetneq S$ is a subsurface with smooth boundary. Then 
\begin{align} \label{eq:b1-bound}
\inbetti R S &\le \frac{|R| + |\partial R|}{2 \pi} + \numcomp(\partial R) \\
&\le \frac{|R|}{2 \pi}+  \Big(1 + \frac1{\sys(S)} \Big) |\partial R|. \label{eq:b1-bound2}
\end{align}
\end{proposition}

\begin{proof}
It suffices to prove \eqref{eq:b1-bound} for each component of $R$, so we will assume that $R$ is connected.
First suppose that $R$ has geodesic boundary. 
By Gauss-Bonnet we have
\begin{equation}\label{eq-kerb}
\inbetti R S = b_1(R) =  -\chi(R) + 1 = \frac{|R|}{2 \pi} + 1\le \frac{|R|}{2 \pi} + \numcomp(R)
\end{equation}
since $\partial{R}$ is nonempty. 

For the case of a general $R$, let $\Omega \subset S$ be convex core of $R$, i.e.\  the surface with geodesic boundary obtained by filling in the holes (contractible components of the complement) of $R$, and replacing each boundary curve of $R$ with its geodesic representative. (Note that we may have $\Omega \not\subset R$.)
On one hand, by Proposition \ref{prop-iso-1} we know that 
\begin{equation}\label{eq-panta}
|\Omega|\le |R|+|\partial{R}|. 
\end{equation}
On the other hand,
$\inbetti R S = \inbetti \Omega S$
because filling in the holes (and homotoping the boundary) does not change the image of $H_1$ in $H_1(S)$. 
Combining this with \eqref {eq-kerb} and \eqref{eq-panta}, 
we get
$$
\inbetti R S\le  \frac{|\Omega|}{2 \pi} + \numcomp(\partial\Omega)
\le   \frac{|R| + |\partial R|}{2 \pi}  + \numcomp(\partial R). \qedhere
$$
\end{proof}

Combining Proposition \ref{prop-druga-new} and Proposition \ref{prop-bound} enables us to estimate the homology of $S(\delta)$.

\begin{proposition}\label{prop-druga-new-1} For every $\eta>0$ there exist $\delta(\eta)>0$,  $\epsilon_0(\eta)>0$,  
such that
for every $\epsilon$-nearly geodesic minimal surface  $f\from S\to \M$,  assuming $\epsilon<\epsilon_0(\eta)$, 
there is a subsurface  (with smooth boundary)
$R_\eta \subset S$ for which
\begin{enumerate}
\item  
$\Ne_{10}(S(\delta(\eta)))\subset  R_\eta \subset S(\delta_0)$, 
\item 
$\inbetti {R_\eta} S \le \frac{|R_\eta|}{2 \pi} +\eta |S|$.
\end{enumerate} 
\end{proposition}

\begin{proof} 
Let $\alpha = \eta/(1 + 2/\sys(\M))$, let $\delta(\eta)$ be the $\delta(\alpha)$ from  Proposition \ref{prop-druga-new},
and let $\epsilon_0(\eta)$ be the minimum of $\epsilon_0(\alpha)$ given by Proposition \ref{prop-druga-new} and a universal $\epsilon_1$ given below. 

Suppose we are given $f\from S \to \M$ that is $\epsilon_0(\eta)$-nearly geodesic, 
with $\epsilon_0(\eta) \le \epsilon_0(\alpha)$. 
Let $R= R(\alpha)$ be the surface from the statement of Proposition \ref{prop-druga-new}. 
First, from (\ref{eq:b1-bound2}) we get
\begin{align}
\inbetti {R} S
&\le \frac{|R|}{2 \pi} + \left(1+\frac{1}{\sys(S)}\right)|\partial R| 
&\le \frac{|R|}{2 \pi} + \left(1+\frac{2}{\sys(\M)}\right)|\partial R| \label{eq-betti-R-eta}
\end{align}
because $\sys(S) \ge \sys(\M)/2$ when $\epsilon_0 \le \epsilon_1$, where $\epsilon_1$ is a universal constant. 
We now apply the estimate  (2) from Proposition \ref{prop-druga-new}, and get
\begin{equation}\label{eq-ziv}
|\partial R| \le \alpha |S|.
\end{equation}
Putting together \eqref{eq-betti-R-eta}, \eqref{eq-ziv}, and the definition of $\alpha$, we obtain Property (2) of this proposition. 
Property 1 follows immediately from Property 1 of Proposition \ref{prop-druga-new}.
\end{proof}

\subsection{Proof of Lemma \ref{lemma-tri}}

Fix  an $\epsilon$-nearly geodesic minimal map $f\from S\to \M$. We may assume $\sys(\M)\le 2\sys(S)$. 
By Theorem \ref{thm-bounded-geom} we can find a bounded geometry triangulation $T$ of $S$. 

Let $\eta=\frac{q}{4\pi}$. Let $R=R_\eta$ be the subsurface from Proposition \ref{prop-druga-new-1}, and let $\delta \equiv \delta(q, h)=\delta(\eta)$.
Then from  (1) in Proposition \ref{prop-druga-new-1} we get
$R \subset S(\delta_0)$\, 
which implies 
\begin{equation}\label{eq-seta}
|R| \le (1-q)|S|.
\end{equation}

We let $T_0$ be the vertices and edges of $T$ that lie entirely in $R$.
We can then extend $T_0$ to a connected subgraph $T_1$ of $T$ that includes all the vertices of $T$; 
any minimal such extension will add no new cycles to those of $T_0$, and hence we'll have $\inbetti {T_1} S = \inbetti {T_0} S$.
We observe that 
\begin{align*}
\inbetti {T_1} S &= \inbetti {T_0} S \le \inbetti R S \\
&\le  \frac{|R|}{2 \pi} +\eta |S|  \le  \frac{(1-q)|S|}{2 \pi} + \eta |S| \\
&\le \big((1-q)+\frac{q}{2} \big)2(\gen-1) \\
&\le (2-q)\gen.
\end{align*}
Here we used  (2) from   Proposition \ref{prop-druga-new-1} in the second inequality, (\ref{eq-seta})  in the third inequality, and the choice of $\eta$ in the fourth.
Thus, 
we have shown 
\begin{equation}\label{eq-gru}
\inbetti {T_1} S \le (2-q)\gen.
\end{equation}

We then let $\wh{T}$ be a spanning tree for $T_1$. Then, 
as in \cite{k-m-1},
we find a set $e(T)$ of $2\gen$ edges of $T \setminus \wh{T}$ such that $\inbetti {e(T) \cup \wh{T}} S = 2\gen$.
Let $e_0(T) \subset e(T)$ be the edges that are contained in $R$. Then $\big(\wh{T} \cup e_0(T)\big) \subset T_1$, and by (\ref{eq-gru}) we have
$$
|e_0(T)| =\inbetti{\wh{T} \cup e_0(T)} S\le \inbetti {T_1} S \le (2-q)\gen.
$$
Letting $e_1(T)$ be $e(T)\setminus e_0(T)$, 
we have 
\begin{equation} \label{eq-e1qg}
|e_1(T)| \ge q\gen.
\end{equation}

We now  let $\wh{\tau}=\wh{T}$, and $e(\tau)=e(T)$.
In light of \eqref{eq-e1qg} we have thus produced a bounded geometry embedding of $\tau$ where at least $q\gen$ distinguished edges have the property required by the Lemma.

\section{Rescaled Riemannian metrics and a cover with wafer-thin disks} \label{section-Dol}
\subsection{A cover with wafer-thin disks}
In this section we use a rescaling of the Riemannian metric on $\G_2(\M)$ to construct the cover $\Dol$ needed for Lemma \ref{lemma-Dol}.

We will write $\G_2$ when the given operation applies to both $\G_2(\M)$ and $\G_2({\Ho})$.
In $\G_2$ we have already defined  an invariant Riemannian metric which we will just denote by $\pair{}{}$.
We also have the tautological foliation $\FF$ of $\G_2$ by hyperbolic planes. 
This gives us an invariant splitting $T\G_2= T\FF \oplus T\FF^\perp$. 
If $v \in T\G_2$, we say that $v$ \emph{respects the splitting} if  $v \in T\FF \cup T\FF^\perp$. 
If $v_0, v_1 \in T\G_2$ respect the splitting, 
we let $m(v_0, v_1)$ be the number of $i \in \setof{0, 1}$ such that $v_i \in T\FF^\perp$. 
We now define $\pair{}{}_\eta$ by 
$$\pair vw_\eta = \eta^{-m(v, w)} \pair vw$$
whenever $v$ and $w$ respect the splitting.
Thus the norm of a vector in $T\FF$ is unchanged, while the norm of a vector in $T\FF^\perp$ is rescaled by $\eta^{-1}$.
We call this Riemannian metric (and associated point-pair metric $d_\eta$) the $\eta$-metric. 
We note that $\pair vw_\eta$ and $d_\eta(p, q)$ are both nondecreasing functions of $\eta$. 
We let $B_r^\eta(p)$ be the $r$-ball in the $\eta$-metric around a pointed plane $p$.
We let $|B_r^\eta(p)|_\eta$ denote the volume of this ball in the $\eta$-metric (when it is embedded), and we observe that
\begin{equation} \label{eq-vol-vol}
|B_r^\eta(p)| = \eta^3|B_r^\eta(p)|_\eta,
\end{equation}
where the left-hand side is the volume of the same ball in the original metric. 

We can prove the following lemma with a simple compactness argument. 
\begin{lemma} \label{lem-hull-bound}
For any $r < r'$
and $\eta > 0$, there exists $\epsilon$ such that for $p \in \G_2$,
\begin{equation} \label{eq-hull-1}
\cl{\Hull_{r, \epsilon}(p)}\subset B_{r'}^\eta(p),
\end{equation}
and hence,
for sufficiently small $\epsilon\equiv\epsilon(r_0,  r_1, r')$,
\begin{equation} \label{eq-hull-2}
\Hull_{r_0, \epsilon}(B_{r_1}^\eta(p)) \subset B_{r'}^\eta(p)
\end{equation}
whenever $r' > r_0 + r_1$. 
\end{lemma}
\begin{proof}
The second statement follows from the first, so we will just prove the first.
Because of the homogeneity, 
we can assume that $p$ is at the origin of ${\Ho}$. 
Suppose that for a given $r < r'$ and $\eta$, there is no $\epsilon$.
Then we have $\bigcap_{\epsilon>0} \cl{\Hull_{r, \epsilon}(p)} \not\subset B_{r'}^\eta(p)$,
because $\Hull_{r, \epsilon}(p)$ is nested as $\epsilon$ decreases to 0, and has compact closure. 
On the other hand,
$$
\bigcap_{\epsilon>0} \cl{\Hull_{r, \epsilon}(p)} = \cl{P_r(p)} \subset B_{r'}^\eta(p). \qedhere
$$
\end{proof}

The following lemma may seem obvious, but requires a little thought to prove, so the proof is included.
\begin{lemma} \label{lem-eta-B-P}
For any $R, \delta > 0$ there exists an $\eta_0 > 0$ such that $B_{r}^\eta(p) \subset \Ne_\delta(P_{r}(p))$ for all $r \le R$ and $\eta \le \eta_0$.  
\end{lemma}
\begin{proof}
We work in $\G_2({\Ho})$; it is then a simple matter to conclude the same for $\G_2(\M)$. 

For a sufficiently small $u$, we can form the radius $u$ orthogonal hyper-disk $V$ to $P_{r}(p)$ at $p$ with the exponential map (for the standard metric) applied to the orthogonal complement to the tangent subspace in $\G_2(\M)$ of $P_{r}(p)$. 
We can define a projection $\pi$ to $V$ on every hyperbolic plane in $\G_2({\Ho})$ that intersects $V$ (where we project each plane to its intersection with $V$). 
We then let $B' = B_R(p) \cap \pi^{-1}(V)$. 
We observe that, by compactness of $\cl{B'}$, there exists $\alpha > 0$ such that $\norm{w} \ge \alpha\norm{D\pi(w)}$ for every vector $w$ based at a point in $B'$ and orthogonal to $\FF$. 
Then for any path $\gamma$ in $B'$,
we have 
\begin{equation}
l_\eta(\gamma) \ge \frac{\alpha}{\eta} l(\pi(\gamma)).
\end{equation}
Therefore%
\footnote{By ``continuous induction''}
any path $\gamma$ starting at $p$ of length at most $R$ lies in $B'$ when $\eta < \alpha u/R$,
and $\pi(\gamma)$ then has length at most $\eta R/\alpha$.
So for $\eta$ sufficiently small,
every point in $B_R^\eta(p)$ is connected to $p$ by a path of length $R$ that lies in the preimage by $\pi$ of a small ball around $p$ in $V$. 
This must then lie in $\Ne_\delta(P_R(p))$%
\footnote{Because the hyperbolic planes in $\G_2({\Ho})$ are locally length-minimizing.}.

For $r \in [\delta, R]$ we can apply the same reasoning, and obtain a uniform $\eta_0$ by compactness.
For $r < \delta$, we can take $\eta_0 = 1$. 
\end{proof}

%
We now proceed to constructing the desired cover  $\Dol$ with balls in the $\eta$-metric. We first observe
\begin{lemma} \label{lem-bounded-sectional}
The sectional curvatures of $\pair{}{}_\eta$ are uniformly bounded for $\eta \in (0, 1]$. 
\end{lemma}

\begin{corollary} \label{cor-volume}
There are functions $c(r)$ and $C(r)$ such that for all $\eta \in (0, 1]$ and $r >0$, 
\begin{equation}
c(r) \le  |B_r^\eta(p)|_\eta \le C(r).
\end{equation}
\end{corollary}
%
The proofs of Lemma \ref{lem-bounded-sectional} and Corollary \ref{cor-volume} involve a modicum of Riemannian geometry and are deferred until Section \ref{subsec-cv}.

\begin{corollary} \label{cor-balls-balls}
For any $r, R \in \R^+$, there is an $N \equiv N(r, R)$ such that for all $\eta \in (0, 1]$, 
any $R$-ball in the $\eta$-metric contains at most $N$ disjoint $r$-balls (in the $\eta$-metric). 
\end{corollary}
\begin{proof}
We let $N(r, R) = C(R) /c(r)$ where $C$ and $c$ are taken from Corollary \ref{cor-volume}.
\end{proof}

We can now easily prove the main estimate for this section:
\begin{theorem} \label{thm-nice-cover}
For all $r, R \in \R^+$, there is an $N \in \Z^+$ such that for all $\eta \in (0, 1]$,
there is a finite set $F \subset \G_2(\M)$ such that $\G_2(\M) \subset \bigcup_{p \in F} B_r^\eta(p)$, 
and every point of $\G_2(\M)$ lies in $B_R^\eta(p)$ for at most $N$ distinct $p \in F$. 
\end{theorem}
\begin{proof}
Given $r, R$ and $\eta$, 
by Lemma 2.6,
we can find a finite set $F \subset \G_2(\M)$ such that $B_{r/2}(p)$ for $p \in F$ are mutually disjoint,
and $\G_2(\M) \subset \bigcup_{p \in F} B_r^\eta(p)$.

For any $q \in \G_2(\M)$, if $q \in B_R^\eta(p)$, 
then $B_{r/2}(p) \subset B_{R + r}(q)$.
Therefore $q$ lies in $B_R^\eta(p)$ for at most $N(r/2, r+R)$ distinct $p \in F$, where $N(\cdot, \cdot)$ is given by Corollary \ref{cor-balls-balls}.
\end{proof}

\newcommand{\sm}{s_\M}
\newcommand{\tm}{t_\M}
\newcommand{\uu}{\mathbf u}

Equipped with this theorem we can now prove Lemma \ref{lemma-Dol}.
\begin{proof}[Proof of Lemma \ref{lemma-Dol}]
We take $r \in (0, 1]$ such that $\diam \xi(B_r(p)) \le r_\M/10$, and let $R = 4$.
We then let $C_{\ref{lemma-Dol}} = N_{\ref{thm-nice-cover}}(r, R)$. 
We then observe that for any $\eta >0$ we can apply Theorem \ref{thm-nice-cover} to obtain a finite cover $\Dol$ 
by disks of the form $B_r^\eta(p)$ (for $p$ in a finite set $F$) 
that satisfies Properties 1 and 2 of Lemma \ref{lemma-Dol} for $\epsilon$ sufficiently small.
Property 1 is immediate. 
To see Property 2,
suppose that 
\begin{equation} \label{eq-hbb}
\Hull_{1, \epsilon}(B_r^\eta(p)) \cap B_r^\eta(q) \neq \emptyset .
\end{equation}
Since $\Hull_{1, \epsilon}(B_r^\eta(p)) \subset B_3^\eta(p)$ for $\epsilon$ sufficiently small,
\eqref{eq-hbb} implies that $B_r^\eta(q) \subset B_4^\eta(p)$. 
But the conclusion of Theorem \ref{thm-nice-cover} is that this can happen, given $p \in F$, for at most $C$ values of $q$ in $F$. 

So now we need only find $h \in \N$, $\eta > 0$, and $\epsilon > 0$ that satisfy Property 3. 
Given $c$, we let $h = \ceil{2/c}$. 
We then claim that for $D = B_r^\eta(p) \not\subset\Ne_\delta(\hh)$, 
and $f\from S \to \M$ an $\epsilon$-nearly geodesic map,
we will have $\mu_f(\cl{\Hull_{1,\epsilon}(D)}) \le c$ for $\eta$ and $\epsilon$ sufficiently small given $\delta > 0$ (and independent of $p$).

Note that $\Hull_{1, \epsilon}(B_r^\eta(p)) \subset B_2^\eta(p)$ when $\epsilon$ is sufficiently small given $\eta$
(we can always choose $\epsilon$ after $\eta$),
and, assume, for the sake of contradiction, that there is a  $\delta >0$ and sequences $\eta_n \to 0$, $p_n \in \G_2(\M) \setminus \Ne_\delta(\hh)$, and  $\epsilon_n$-nearly geodesic maps $f_n\from S_n \to \M$, where $\epsilon_n \to 0$, such that 
$$
\mu_{f_n}(B_2^{\eta_n}(p_n)) \ge c.
$$   
Passing to a subsequence, we may assume that $\mu_{f_n}\to \mu$ and $p_n \to p$ with $p \notin \hh$; 
by Lemma \ref{lem-eta-B-P}, $\cl{B_2^{\eta_n}(p_n)} \to \cl{P_2(p)}$ in the Hausdorff metric. 
Letting $P = P_2(p)$ we then have
\begin{equation}
\mu(P) \ge c.
\end{equation}
Since $\mu$ is $\PSLR$ invariant, 
by Ratner's theorem we can write $\mu=\mu_{\hh} + \mu_{\hhh}+\mu_\mathcal{L}$, 
where $\mu_{\hh}$ and $\mu_{\hhh}$ are supported on $\hh$ and $\hhh$ respectively, 
and $\mu_\mathcal{L}$ is a multiple of the Liouville measure. 
Since $p \notin \hh$, we must have $P \cap \hh = \emptyset$, and hence 
\begin{equation} \label{eq-le}
\mu_{\hh}(P) = 0.
\end{equation}
We also have
\begin{equation} \label{eq-L}
\mu_\mathcal{L}(P) = 0 
\end{equation}
because $P$ is infinitely thin.
Finally,
letting $|P| = 2\pi(\cosh 2 - 1) < 8\pi$ be the area of the hyperbolic disk of radius 2, we have
\begin{equation} \label{eq-gg}
\mu_{\hhh}(P) \le \frac{|P|}{4\pi h} < 2/h \le c
\end{equation}
because $P$ can only intersect one totally geodesic surface in $\hhh$, 
and it represents at most the given fraction of the total area of that surface. 
Combining \eqref{eq-le}, \eqref{eq-L}, and \eqref{eq-gg}, we obtain $\mu(D) < c$, a contradiction. 
\end{proof}

\subsection{Estimates of curvature and volume for the $\eta$-metric}
\label{subsec-cv}
We observe that around every point $p$ of $\G_2$ we can find an orthonormal basis $(e_i)$ of vector fields in $T \G_2$ of a neighborhood $U$ of $p$ that respect the splitting. (This basis will not actually be invariant under the isometry group of ${\Ho}$, but this will not be a problem.) We then define the \emph{structure functions} $\alpha_{ijk}$ by $\alpha_{ijk} = \pair {[e_i, e_j]}{e_k}$. Because $\pair{e_i}{e_j}=\delta_{ij}$ is constant, the Koszul formula for $\pair{\nabla_{e_i}{e_j}}{e_k}$ will just be a sum and difference of the $\alpha_{ijk}$. It then follows that the 
curvature tensor $R^{ijkl}$ is bounded when the $\alpha_{ijk}$ are, and then the sectional curvatures are as well. 

So, we take $e_i^\eta = e_i$ for $e_i \in T\FF$, and $e_i^\eta = \eta e_i$ for $e_i \in T\FF^\perp$, so that $(e_i^\eta)$ is an  invariant orthonormal basis for the $\eta$-metric (and respects the splitting). We let $\alpha_{ijk}^\eta$ be the corresponding structure constants. 
We then observe (letting $m(i, j) = m(e_i, e_j)$)
\begin{enumerate}
\item
$\alpha^\eta_{ijk} = 0$ when $e_i, e_j \in T\FF$, and $e_k \in T\FF^\perp$,
\item
$\alpha^\eta_{ijk} = \eta^{m(i, j)-1} \alpha_{ijk}$ when $e_k \in T\FF^\perp$, and
\item
$\alpha^\eta_{ijk} = \eta^{m(i, j)} \alpha_{ijk}$ when $e_k \in T\FF$. 
\end{enumerate}
Hence the exponent for $\eta$ is non-negative whenever $\alpha_{ijk}$ is non-zero.
This completes the proof of Lemma \ref{lem-bounded-sectional}.

Corollary \ref{cor-volume} follows immediately from Lemma \ref{lem-bounded-sectional} and the following comparison theorem (see Corollary 3 in Section 11.9 of the book \cite{b-c}).
\begin{theorem} \label{thm-vol-compare}
Suppose all sectional curvatures have absolute value less than or equal to $K$.
Then the volume of any $r$-ball $B_r(p)$ satisfies
$$
|B_r(p)| \le V(r, -K),
$$
and also, for $r < \pi/K$, 
$$
|B_r(p)| \ge V(r, K),
$$
where $V(r, K)$ denotes the volume of the $r$-ball in constant curvature $K$. 
\end{theorem}
Theorem \ref{thm-vol-compare} is an immediate corollary to the Rauch comparison theorem. 
More precisely, Theorem \ref{thm-vol-compare} follows from the analogous comparison theorem for $r$-spheres in the two metrics;
this in turn follows from the analogous comparison between the volume forms of the  Riemannian metrics on the $r$-spheres in the natural coordinates (parametrized by the unit spheres in the tangent spaces at the centers of the balls).
The comparison of the volume forms follows from the comparison of the norms, and this is literally a special case of the Rauch comparison theorem.

\end{document}